\documentclass[leqno,11pt,draft]{amsart}
\usepackage{amssymb}
\usepackage{amsmath}
\usepackage{enumerate}
\usepackage{amsfonts}
\usepackage{epic}

\newtheorem{teo}{Theorem}[section]

\newtheorem{lema}{Lemma}[section]

\begin{document}

\title[ultraspherical Riesz transforms]%
{Higher order Riesz transforms in the ultraspherical setting as
principal value integral operators}

 \keywords{}
\begin{abstract}
In this paper we represent the $k$-th Riesz transform in the
ultraspherical setting as a principal value integral operator for
every $k \in \mathbb N$. We also measure the speed of convergence
of the limit by proving $L^p$-boundedness properties for the
oscillation and variation operators associated with the
corresponding truncated operators.

\end{abstract}

\author[J. Betancor]{Jorge J. Betancor}
\address{Departamento de An\'{a}lisis Matem\'{a}tico\\
Universidad de la Laguna\\
Campus de Anchieta, Avda. Astrof\'{\i}sico Francisco S\'{a}nchez, s/n\\
38271 La Laguna (Sta. Cruz de Tenerife), Spain}
\email{jbetanco@ull.es; jcfarina@ull.es; lrguez@ull.es}

\author[J.C. Fari\~{n}a]{Juan C. Fari\~{n}a}
%\address{Departamento de An\'{a}lisis Matem\'{a}tico\\
%Universidad de la Laguna\\
%Campus de Anchieta, Avda. Astrof\'{\i}sico Francisco S\'{a}nchez, s/n\\
%38271 La Laguna (Sta. Cruz de Tenerife), Spain} \email{}

\author[L. Rodr\'{\i}guez-Mesa]{Lourdes Rodr\'{\i}guez-Mesa}
%\address{Departamento de An\'{a}lisis Matem\'{a}tico\\
%Universidad de la Laguna\\
%Campus de Anchieta, Avda. Astrof\'{\i}sico Francisco S\'{a}nchez, s/n\\
%38271 La Laguna (Sta. Cruz de Tenerife), Spain} \email{}

\author[R. Testoni]{Ricardo Testoni}
\address{Departamento de Matem\'atica\\
Universidad Nacional del Sur\\
Avda. Alem 1253 - 2º Piso\\
8000 Bah\'{\i}a Blanca, Buenos Aires, Argentina}
\email{ricardo.testoni@uns.edu.ar}

\thanks{This paper is partially supported by MTM2007/65609.}
\date{\today}

\maketitle

\section{Introduction}

Muckenhoupt and Stein \cite{MS} introduced a notion of conjugate
functions associated with ultraspherical expansions. In this setting
the conjugate function appears as a boundary value of a conjugate
harmonic extension associated with a suitable Cauchy-Riemann type
equations.

Assume that $\lambda >0$. For every $n \in \mathbb N$, we denote by
$P_n^\lambda$ the ultraspherical polynomial of degree $n$
(\cite{Sz}). These polynomials are defined by the generating
relation
$$
(1-2tw+w^2)^{-\lambda} = \sum^\infty_{k=0} w^k P_k^\lambda(t).
$$
The sequence $\{P_n^\lambda(\cos\theta)\}_{n\in \mathbb N}$ is
orthogonal and complete in the space 
$L^2((0,\pi),dm_\lambda (\theta ))$, where $dm_\lambda (\theta )=(\sin\theta)^{2\lambda}d\theta$. When $2\lambda=k-2$,
with $k \in \mathbb N$, the $\lambda$-ultraspherical polynomial
$P_n^\lambda$, $n\in \mathbb N$, arises in the Fourier analysis of
functions in the surface of the $n$-Euclidean space sphere that are
invariant under the rotations leaving a given axis fixed.

For every $n \in \mathbb N$, $P_n^\lambda(\cos\theta)$ is an
eigenfunction of the operator
$$
L_\lambda=-\frac{d^2}{d\theta^2}-2\lambda \cot \theta \frac{d}{d\theta}+ \lambda^2,
$$
associated with the eigenvalue $\mu_n=(n+\lambda)^2$. The operator
$L_\lambda$ can be written as follows
$$
L_\lambda=-\left(\frac{d}{d\theta}\right)^*\frac{d}{d\theta} + \lambda^2,
$$
where $\displaystyle\left(\frac{d}{d\theta}\right)^*=\frac{d}{d\theta}+2\lambda\cot\theta$ denotes the formal adjoint of $\displaystyle\frac{d}{d\theta}$ in 
$L^2((0,\pi),dm_\lambda (\theta ))$.

In \cite{BMTU} Buraczewski, Mart{\'\i}nez, Torrea and Urban defined
a Riesz transform in the ultraspherical setting associated to
$L_\lambda$. Note that this operator $L_\lambda$ is slightly
different than the one considered by Muckenhoupt and Stein (see
\cite[p. 23]{MS}). In \cite{BMTU} the authors follow the ideas
developed in the monography of Stein \cite{Stein}.

Suppose that $f \in L^2((0,\pi),dm_\lambda (\theta ))$. The ultraspherical expansion of $f$ is
$$
f(\theta)=\sum^\infty_{n=0} a_n^\lambda(f)\frac{P_n^\lambda(\cos\theta)}{\|P_n^\lambda(\cos\cdot)\|}_{L^2((0,\pi),dm_\lambda (\theta ))},
$$
where, for every $n \in \mathbb N$,
$$
a_n^\lambda(f)= \int_0^\pi
f(\theta)\frac{P_n^\lambda(\cos\theta)}{\|P_n^\lambda(\cos\cdot)\|_{L^2((0,\pi),dm_\lambda (\theta ))}}dm_\lambda (\theta ).
$$
The Poisson integral $P_t^\lambda(f)$, $t >0$, is given by
$$
P_t^\lambda(f)(\theta)=
e^{-t\sqrt{L_\lambda}}f(\theta)=\sum^\infty_{n=0}a_n^\lambda(f)e^{-t(n+\lambda)}
\frac{P_n^\lambda(\cos\theta)}{\|P_n^\lambda(\cos\cdot)\|_{L^2((0,\pi),dm_\lambda (\theta ))}},\quad
t>0.
$$
According to \cite[(2.12)]{MS} we can write
\begin{equation}\label{Poisson}
P_t^\lambda f(\theta) = \int_0^\pi r^\lambda
P_\lambda(e^{-t},\theta,\varphi)f(\varphi)dm_\lambda (\varphi ),\quad t>0,
\end{equation}
where, for each $0<r<1$ and $\theta,\varphi\in (0,\pi)$,
$$
P_\lambda(r,\theta,\varphi)=
\frac{\lambda}{\pi}(1-r^2)\int^\pi_0\frac{\sin^{2\lambda-1}t}{(1-2r(\cos\theta\cos\varphi+\sin
\theta \sin\varphi\cos t) + r^2)^{\lambda+1}}dt.
$$
The $L^p$-boundedness properties for these Poisson integrals and the
corresponding maximal operator were established in \cite[Theorem
2.4]{BMTU} (see also \cite[Theorem 2]{MS}). For every $\alpha >0$,
the fractional power $L_\lambda^{-\alpha}$ of the operator
$L_\lambda$ is defined by
$$
L^{-\alpha}_\lambda f(\theta)=\frac{1}{\Gamma(2\alpha)}\int_0^\infty e^{-t\sqrt{L_\lambda}}f(\theta) t^{2\alpha -1}dt,\quad f \in L^2((0,\pi),dm_\lambda (\theta )).
$$
By using (\ref{Poisson}) we get, for every $\alpha >0$ and $f\in
L^2((0,\pi),dm_\lambda (\theta ))$,
\begin{eqnarray*}
L_\lambda^{-\alpha}f(\theta) &=& \frac{1}{\Gamma(2\alpha)}
\int_0^{\infty}\int_0^\pi
P^\lambda(e^{-t},\theta,\varphi)f(\varphi)dm_\lambda (\varphi )
t^{2\alpha-1} dt \\
&=& \int_0^\pi
f(\varphi)\frac{1}{\Gamma(2\alpha)}\int_0^1P^\lambda(r,\theta,\varphi)\left(\log\frac{1}{r}\right)^{2\alpha-1}\frac{1}{r}drdm_\lambda (\varphi ).
\end{eqnarray*}
Following \cite{Stein} the Riesz transform of order $k\in \mathbb
N$, $R_\lambda^k$, is defined as
$$
R_\lambda^kf= \frac{d^k}{d\theta^k}L_\lambda^{-\frac{k}{2}}f,
$$
when $f$ is a nice function (for instance, $f \in \mbox{span}
\{P_n^\lambda(\cos\theta)\}_{n\in \mathbb{N}}$ or $f$ is a smooth
function with compact support on $(0,\pi)$).

It was proved in \cite[Theorem 2.14]{BMTU} (when $k=1$) and
\cite[Theorem 1.4]{BMT} (when $k>1$) that the operator $R_\lambda^k$
can be extended to
$L^p((0,\pi),w(\theta)dm_\lambda (\theta ))$ as a bounded
operator from $L^p((0,\pi),w(\theta)dm_\lambda (\theta ))$
into itself, when $1<p<\infty$ and $w \in A_\lambda^p$, and as a
bounded operator from
$L^1((0,\pi),w(\theta)dm_\lambda (\theta ))$ into
$L^{1,\infty}((0,\pi),w(\theta)dm_\lambda (\theta ))$,
when $w \in A_{\lambda}^1.$ Here, for every $1 \leq p < \infty$, by
$A_\lambda^p$ we denote the Muckenhoupt class of weights associated
with the doubling measure $dm_\lambda (\theta )$ on
$(0,\pi)$.

In this paper we prove that the $k$-th Riesz transform $R^k_\lambda$ is a principal value integral operator, for every $k\in \mathbb N$. We extend \cite[Theorem 2.13]{BMTU} where the result is shown for $k=1$.
\begin{teo}\label{thprin}
Let $\lambda >0$ and $k\in \mathbb N$. For every $1\leq p < \infty$ and $\omega \in A_\lambda^p$, we have that if $f \in L^p((0,\pi),w(\theta)dm_\lambda (\theta ))$
\begin{equation}\label{A1}
R_\lambda^kf(\theta)= \lim_{\varepsilon \rightarrow
0^+}\int^\pi_{0,\,|\theta-\varphi|>\varepsilon}
R_\lambda^k(\theta,\varphi)f(\varphi)dm_\lambda (\varphi )
+ \gamma_kf(\theta),\quad \mbox{a.e. }\theta \in (0,\pi),
\end{equation}
where
$$
R_\lambda^k(\theta,\varphi)= \frac{1}{\Gamma(k)}\int_0^1
\frac{\partial^k}{\partial \theta^k}P_\lambda(r,\theta,\varphi)
\left(\log\frac{1}{r}\right)^{k-1} r^{\lambda-1}dr,\quad
\theta,\varphi \in (0,\pi),
$$
and $\gamma_k=0$, when $k$ is odd, and $\gamma_k=(-1)^{\frac{k}{2}}$,
when $k$ is even.
\end{teo}

The complete proof of this theorem is presented in Section 2. It is
a crucial point in the proof the estimates established in Lemma
\ref{lema} below. We prove in this lemma that in the local region,
that is, close to the diagonal $\{\theta=\varphi\}$, the kernel
$R_\lambda^k(\theta,\varphi)$ differs from the kernel of the $k$-th
Euclidean Riesz transform by an integrable function. Also, we show
that far from the diagonal $R_\lambda^k(\theta,\varphi)$ is bounded
by Hardy type kernels.

Suppose that $\{T_\varepsilon\}_{\varepsilon >0}$ is a family of
operators defined on $L^p(\Omega,\mu)$, for some measure space
$(\Omega,\mu)$ and $1 \leq p < \infty$, such that for every $f \in
L^p(\Omega,\mu)$ there exists $\lim_{\varepsilon \rightarrow 0^+}T_\varepsilon
f(x)$, $\mu$-a.e. $x\in \Omega$. It is an interesting
question to measure the speed of that convergence. In order to do
this it is usual to analyze expressions involving differences like
$|T_\varepsilon f-T_\eta f|$, $\varepsilon, \eta
>0$. The oscillation and variation operators defined as follows
have been used for this purpose. The oscillation operator associated
with $\{T_\varepsilon\}_{\varepsilon>0}$ is defined by
$$
O\left(\{T_\varepsilon\}\right)(f)(x)= \left(\sum^\infty_{i=0}
\sup_{t_{i+1} \leq \varepsilon_{i+1} <\varepsilon_i<t_i}
|T_{\varepsilon_{i+1}}f(x) -
T_{\varepsilon_i}f(x)|^2\right)^{\frac{1}{2}},
$$
for a fixed real sequence $\{t_i\}_{i \in \mathbb N}$ decreasing to zero.
For every $\rho >2$ the $\rho$-variation operator for
$\{T_\varepsilon\}_{\varepsilon>0}$ is given as follows
$$
V_\rho\left(\{T_\varepsilon\}\right)(f)(x)= \sup_{\{\varepsilon_i\}_{i \in \mathbb N}}\left(\sum^\infty_{i=0} |T_{\varepsilon_{i+1}}f(x) - T_{\varepsilon_i}f(x)|^\rho\right)^{\frac{1}{\rho}},
$$
where the supremum is taken over all real sequences 
$\{\varepsilon_i\}_{i \in \mathbb N}$ decreasing to zero. These
operators appear in an ergodic context. In \cite{CJRW}, \cite{CJRW1}
and \cite{RW} the $L^p$-boundedness properties for the oscillation
and variation operators were studied when
$T_\varepsilon,\;\varepsilon
>0$, represents the truncated Hilbert and higher dimensional Riesz
transform, and Euclidean Poisson semigroup (see also \cite{GilTo}
and the references therein). The corresponding results for the
truncated ultraspherical Riesz transform were established in
\cite[Theorem 8.3]{BMTU}. We also measure the speed of convergence in
(\ref{A1}) in terms of variation and oscillation operators for
the corresponding truncated operators. Next result is an extension
of \cite[Theorem 8.3]{BMTU} for the higher Riesz transform
$R_\lambda^k$.
\begin{teo}\label{oscilacion}
Let $\lambda >0$ and $k \in \mathbb N$. For every $\varepsilon >0$ we define by $R_{\lambda,\varepsilon}^k$ the $\varepsilon$-truncation of $R_\lambda^k$ as follows
$$
R_{\lambda,\varepsilon}^k(f)(\theta)=\int_{0,|\theta-\varphi|>\varepsilon}^\pi
R_{\lambda}^k(\theta,\varphi)f(\varphi)dm_\lambda (\varphi ).
$$
If $\{t_i\}_{i\in\mathbb N}$ is a real decreasing sequence that
converges to zero, the oscillation operator
$O(\{R^k_{\lambda,\varepsilon}\})$ is a bounded operator from
$L^p((0,\pi),dm_\lambda (\varphi ))$ into itself, for
every $1<p<\infty$, and from
$L^1((0,\pi),dm_\lambda (\varphi ))$ to
$L^{1,\infty}((0,\pi),dm_\lambda (\varphi ))$. Also,
for every $\rho >2$, the variation operator
$V_\rho(\{R^k_{\lambda,\varepsilon}\})$ is bounded from
$L^p((0,\pi),dm_\lambda (\varphi ))$ into itself, for
every $1<p<\infty$, and from
$L^1((0,\pi),dm_\lambda (\varphi ))$ to
$L^{1,\infty}((0,\pi),dm_\lambda (\varphi ))$.
\end{teo}
We remark that the representation of the $k$-th Riesz transform
$R_\lambda^k$ as a principal value integral operator will allow us
to investigate weighted norm inequalities for $R_\lambda^k$
involving a class of weights wider than the Muckenhoupt class
considered in \cite{BMT}. This question will be studied in a
forthcoming paper.

Troughout this paper by $C$ we always denote a positive constant
that can change from one line to the other one and $i,j$ represent
nonnegative integers.

\section{Proof of Theorem \ref{thprin}}

In \cite[Theorem 1.5]{BMT} it was established that, for every $k\in
\mathbb{N}$, the k-th Riesz transform $R_\lambda^k$ is a
Calder\'on-Zygmund operator in the homogeneous type space
$((0,\pi),|.|,dm_\lambda(\theta))$.
Then, according to \cite[Theorem 9.4.5]{Gra} the maximal operator
$R_{\lambda,*}^k$ given by
$$
R_{\lambda,*}^k(f)=\sup_{\varepsilon>0}|R_{\lambda,\varepsilon}^k(f)|,
$$
where $R_{\lambda,\varepsilon}^k$ is defined as in Theorem
\ref{oscilacion}, is bounded from
$L^p((0,\pi),w(\theta)dm_\lambda (\theta ))$ into itself,
when $1<p<\infty$ and $w \in A_\lambda^p$, and from
$L^1((0,\pi),w(\theta)dm_\lambda (\theta ))$ into
$L^{1,\infty}((0,\pi),w(\theta)dm_\lambda (\theta ))$,
when $w \in A_{\lambda}^1.$ Suppose we have proved that, for every
$f\in C_c^\infty(0,\pi)$, there exists the limit
\begin{equation}\label{limite}
T_\lambda^k(f)(\theta)=\lim_{\varepsilon \rightarrow
0^+}\int^\pi_{0,\,|\theta-\varphi|>\varepsilon}
R_\lambda^k(\theta,\varphi)f(\varphi)dm_\lambda (\varphi ),\quad \mbox{a.e. }\theta\in
(0,\pi),
\end{equation}
and that $T_\lambda^kf=R_\lambda^kf-\gamma_kf$. Then,
$L^p$-boundedness properties of the maximal operator
$R_{\lambda,*}^k$ imply that the limit in (\ref{limite}) exists for
almost all $\theta\in (0,\pi)$, for every $f\in
L^p((0,\pi),w(\theta)dm_\lambda (\theta ))$, $1\le
p<\infty$, and $w\in A_\lambda^p$. Moreover, by defining
$T_\lambda^k$ in the obvious way on
$L^p\!(\!(0,\pi),\!w(\theta)dm_\lambda (\theta)\!)$, $1\le
p<\infty$, $T_\lambda^k$ is a bounded operator from
$L^p((0,\pi),w(\theta)dm_\lambda (\theta ))$ into itself,
when $1<p<\infty$ and $w \in A_\lambda^p$, and from
$L^1((0,\pi),w(\theta)dm_\lambda (\theta ))$ into
$L^{1,\infty}((0,\pi),w(\theta)dm_\lambda (\theta ))$,
when $w \in A_{\lambda}^1.$ Hence, by \cite[Theorem 1.4]{BMT}, we
conclude that, for each $f\in
L^p((0,\pi),w(\theta)dm_\lambda (\theta ))$, $1\le
p<\infty$,
$$
R_\lambda^k(f)(\theta)=\lim_{\varepsilon \rightarrow
0^+}\int^\pi_{0,\,|\theta-\varphi|>\varepsilon}
R_\lambda^k(\theta,\varphi)f(\varphi)dm_\lambda (\varphi )+\gamma_kf(\theta),\quad \mbox{ a.e. }\theta\in
(0,\pi),
$$
and the proof of this theorem would be finished.

Let now $f\in C^\infty_c(0,\pi)$ and $k \in \mathbb N$ . We can
write
$$
L_\lambda^{-\frac{k}{2}}f(\theta)=\sum_{n=0}^\infty
(n+\lambda)^{-k}a_n^\lambda(f)\frac{P_n^\lambda(\cos\theta)}{\|P_n^\lambda(\cos\cdot)\|}_{L^2((0,\pi),dm_\lambda (\theta ))},\quad \theta\in
(0,\pi).
$$
Then, since $f\in C_c^\infty(0,\pi)$, $L_\lambda^{-\frac{k}{2}}f\in
C^\infty(0,\pi)$ (see \cite[(2.4) and (2.6)]{Mu}). We will see that
$$
\frac{d^k}{d\theta^k}  L_\lambda^{-\frac{k}{2}}f(\theta) =
\lim_{\varepsilon \rightarrow 0^+}\int_{0,|\theta - \varphi|>
\varepsilon}^\pi
f(\varphi)R_\lambda^k(\theta,\varphi)\;dm_\lambda (\varphi )+\gamma_kf(\theta),\quad \mbox{a.e. }
\theta \in (0,\pi),
$$
where
$$
R_\lambda^k(\theta,\varphi)=\frac{\partial^k}{\partial
\theta^k}\left(\frac{1}{\Gamma(k)} \int_0^1r^{\lambda
-1}\left(\log\frac{1}{r}\right)^{k-1}P_\lambda(r,\theta,\varphi)\;dr\right),\quad
\theta,\varphi\in(0,\pi),
$$
$$
P_\lambda(r,\theta,\varphi)= \frac{\lambda }{\pi}
\int_0^\pi\frac{(1-r^2)(\sin
t)^{2\lambda-1}}{(1-2r(\cos\theta\cos\varphi+\sin\theta\sin\varphi\cos
t)+r^2)^{\lambda+1}}dt,\quad r \in (0,1), \theta,\varphi\in(0,\pi),
$$
and $\gamma_k=0$, when $k$ is odd, and $\gamma_k= (-1)^{\frac{k}{2}}$,
when $k$ is even.

As in \cite{BMT} we introduce the following useful notation:
$$
\begin{array}{l}
\sigma=\sin\theta\sin\varphi,\;\;a=\cos\theta\cos\varphi+\sigma\cos t=\cos(\theta-\varphi)-\sigma(1-\cos t), \\
b=\frac{\partial}{\partial \theta} a= -\sin\theta\cos\varphi+\cos\theta\sin\varphi\cos t=-\sin(\theta-\varphi)-\cos\theta\sin\varphi(1-\cos t),\\
\Delta_r=1-2r\cos(\theta-\varphi)+r^2=(1-r)^2+2r(1-\cos(\theta-\varphi)),\;\Delta=\Delta_1, \\
D_r=1-2ra+r^2=\Delta_r+2r\sigma(1-\cos t).
\end{array}
$$
We divide the proof in several steps.

\noindent {\bf Step 1.}  We prove in the following that for every
$\ell=0,\ldots,k-1$,
\begin{equation}\label{lderivada}
\frac{d^\ell}{d\theta^\ell} L_\lambda^{-\frac{k}{2}}f(\theta)= \int_0^\pi f(\varphi)R_\lambda^{k,\ell}(\theta,\varphi)\;dm_\lambda(\varphi),\quad \theta\in(0,\pi),
\end{equation}
being
$$
R^{k,\ell}_\lambda(\theta,\varphi)= \frac{1}{\Gamma(k)}
\int_0^1r^{\lambda
-1}\left(\log\frac{1}{r}\right)^{k-1}\frac{\partial^\ell}{\partial
\theta^\ell} P_\lambda(r,\theta,\varphi)\;dr,\quad \theta,\varphi
\in (0,\pi).
$$

Let $\ell \in \mathbb N$, $0\leq \ell\leq k-1$. Our objective is to establish that
\begin{equation}\label{acotR}
\left|R_\lambda^{k,\ell}(\theta,\varphi)\right| \leq C\left\{\begin{array}{ll}
                                                                    (\sin\varphi)^{-2\lambda-1}, & \hspace{1cm} (\theta,\varphi) \in A_1;\vspace{3mm}\\
                                                                    \displaystyle\frac{1}{(\sin\theta\sin\varphi)^\lambda \sqrt{|\theta-\varphi|}}, & \hspace{1cm} (\theta,\varphi)\in A_2,\,\,\theta\neq\varphi; \vspace{3mm}\\
                                                                    (\sin\theta)^{-2\lambda-1}, & \hspace{1cm} (\theta,\varphi) \in
                                                                    A_3;
                                                                  \end{array} \right.
\end{equation}
where $A_i, i=1,2,3$, are the sets in the next figure:
\begin{center}
\begin{picture}(200,200)
\put(0,10){\line(1,0){200}}
\put(10,0){\line(0,1){190}}
\put(50,10){\circle{2}}
\put(10,50){\circle{2}}
\put(10,90){\circle{2}}
\put(10,130){\circle{2}}
\put(10,170){\circle{2}}
\put(90,10){\circle{2}}
\put(130,10){\circle{2}}
\put(170,10){\circle{2}}
\put(0,50){$\frac{\pi}{4}$}
\put(0,90){$\frac{\pi}{2}$}
\put(0,130){$\frac{3\pi}{4}$}
\put(0,170){$\pi$}
\put(50,0){$\frac{\pi}{4}$}
\put(90,0){$\frac{\pi}{2}$}
\put(130,0){$\frac{3\pi}{4}$}
\put(170,0){$\pi$}
\drawline[0](10,10)(90,50)(170,170)
\drawline[0](10,10)(90,130)(170,170)
\dashline[0]{2}(90,10)(90,50)
\dashline[0]{2}(90,130)(90,170)
\dashline[0]{2}(10,170)(190,170)
\dashline[0]{2}(170,10)(170,190)
\put(50,130){{\bf $A_1$}}
\put(130,50){{\bf $A_1$}}
\put(90,90){{\bf $A_2$}}
\put(50,20){{\bf $A_3$}}
\put(120,155){{\bf $A_3$}}
\put(0,190){$\varphi$}
\put(190,0){$\theta$}
\put(75, -20){Figure 1}
\end{picture}
\end{center}
\vspace*{0.5cm}

According to \cite[Lemma 3.5]{BMT} we have that
\begin{equation}\label{derivDr}
\frac{\partial^\ell}{\partial \theta^\ell}
\left(\frac{1}{D_r^{\lambda+1}}\right)=
\sum_{s,i,j}c_{\ell,s,i,j}\frac{r^{i+j}a^ib^j}{D_r^{\lambda+1+s}},
\end{equation}
where $c_{\ell,s,i,j}\neq 0$ only if
\begin{equation}\label{indices}
s=1,\ldots,\ell,\;\;j \geq 2s-\ell, \mbox{ and }i+j=s.
\end{equation}
Moreover, by using the di Fa\`a di Bruno's formula (\cite[Theorem 2]{Ro}), we can see that, for every $\ell \in \mathbb{N}$, $s=1,...,\ell$, and  $i+j=s$,
\begin{equation}\label{FaadiBruno}
c_{\ell ,s,i,j}=2^s\ell !s!\sum\frac{(-1)^{s+j+\alpha (k_1,...,k_\ell)}}{k_1!\cdots k_\ell !1!^{k_1}2!^{k_2}\cdots \ell !^{k_\ell}},
\end{equation}
where the sum is over all different solutions in nonnegative integers $k_1,...,k_\ell$ of the system
$$
\left.\begin{array}{l}
k_1+k_2+\cdots +k_\ell=s\\
k_1+2k_2+...+\ell k_\ell=\ell\\
\sum_{r\;{ \rm par}}k_r=i\\
\sum_{r\;{\rm impar}}k_r=j
\end{array}\right\}
$$
and
$$
\alpha (k_1,...,k_\ell)=\sum_{r=2}^{[\frac{\ell }{2}]}(r-1)(k_{2r-1}+k_{2r})+m_\ell,
$$being $m_\ell =0$, if $\ell $ is even, and $m_\ell =\frac{(\ell -1)k_\ell}{2}$, when $\ell $ is odd.

We define, for every $s,i,j$ satisfying (\ref{indices}),
$$
M_{\ell,s,i,j}(\theta,\varphi)=\int_0^1\int_0^\pi r^{i+j+\lambda-1}\left(\log\frac{1}{r}\right)^{k-1}(1-r^2)\frac{a^ib^j(\sin t)^{2\lambda-1}}{D_r^{\lambda+1+s}}dtdr,\quad \theta, \varphi \in (0,\pi ).
$$

In order to obtain (\ref{acotR}) it is then sufficient to see that
\begin{equation}\label{acotM}
\left|M_{\ell,s,i,j}(\theta,\varphi)\right| \leq C\left\{\begin{array}{ll}
                                                                    (\sin\varphi)^{-2\lambda-1}, & \hspace{1cm}(\theta,\varphi) \in A_1;\vspace{3mm}\\
                                                                    \displaystyle\frac{1}{(\sin\theta\sin\varphi)^\lambda\sqrt{|\theta-\varphi|}}, &\hspace{1cm} (\theta,\varphi) \in A_2,\,\,\theta\neq\varphi;\vspace{3mm}\\
                                                                    (\sin\theta)^{-2\lambda-1}, & \hspace{1cm}\displaystyle (\theta,\varphi) \in
                                                                    A_3;
                                                                  \end{array} \right.
\end{equation}
for each $s,i,j$ satisfying (\ref{indices}).

Moreover, by the symmetry of the Figure 1 and since
$$
M_{\ell,s,i,j}(\pi -\theta ,\pi -\varphi )=(-1)^jM_{\ell , s,i,j}(\theta , \varphi ),\quad \theta ,\varphi \in (0,\pi ),
$$
when $s,i,j$ are as in (\ref{indices}), we can assume that $(\theta, \varphi )\in (0,\frac{\pi }{2})\times (0,\pi )$.

Let us fix $s,i,j$ verifying (\ref{indices}), $\theta \in (0,\frac{\pi }{2})$ and $\varphi \in (0,\pi )$.
By proceeding as in \cite[Lemma 3.6]{BMT}, and using that $\log\frac{1}{r} \sim 1-r$, as $r\rightarrow 1^-$, and $D_r \geq C$, $r\in (0,\frac{1}{2})$,  we get
\begin{eqnarray*}
|M_{\ell,s,i,j}(\theta,\varphi)| & \leq & \left(\int_0^{\frac{1}{2}}\int_0^\pi + \int_{\frac{1}{2}}^1\int_0^\pi\right)r^{i+j+\lambda-1}\left(\log\frac{1}{r}\right)^{k-1}(1-r^2)\frac{|a|^i|b|^j(\sin t)^{2\lambda-1}}{D_r^{\lambda+1+s}}\;dtdr \\
&\leq& C\left(\int_0^{\frac{1}{2}}r^{\lambda -1}\left(\log\frac{1}{r}\right)^{k-1}dr+\int_{\frac{1}{2}}^1\int_0^\pi (1-r)^k \frac{|b|^j(\sin t)^{2\lambda-1}}{D_r^{\lambda+1+s}}\;dtdr\right)\\
&\leq &C\left(1+\int_{\frac{1}{2}}^1\int_0^\pi (1-r)^k
\frac{|b|^j(\sin t)^{2\lambda-1}}{D_r^{\lambda+1+s}}\;dtdr\right).
\end{eqnarray*}

It can be seen that, if $(\alpha ,\beta )\in [0,\pi ]\times
[0,\pi]$, $z\in (0,1)$ and $\alpha \leq z\beta$, then there exists
$C>0$ such that $\sin (\beta -\alpha )\geq \min \{\sin \beta, \sin
((1-z)\beta )\}\geq C\sin \beta $, and that, if  $(\alpha ,\beta
)\in [0,\pi /2]\times [0,\pi]$ and $\frac{\alpha}{2}\leq \beta\leq
\frac{3\alpha }{2}$, then $\sin |\beta -\alpha |\leq \sin \alpha $
and $\sin \alpha\sim \sin \beta $. These considerations allow us to
write
\begin{equation}\label{bj}
|b|^j\leq C(|\sin (\theta -\varphi )|^j+(\sin \varphi )^j)\leq
C\left\{\begin{array}{ll}|\sin (\theta -\varphi )|^j,&\varphi \leq
        \frac{\theta }{2}\mbox{ or }\varphi \geq \frac{3\theta}{2},\\[0.2cm]
        (\sin \varphi )^j,&\frac{\theta}{2} \le \varphi \le \frac{3\theta}{2}.
        \end{array}
\right.
\end{equation}

Then, since $1-\cos \alpha \geq (\sin \alpha )^2/\pi$, $\alpha \in
[0,\pi]$, we obtain, when $\varphi \leq \frac{\theta }{2}$ or
$\varphi \geq \frac{3\theta}{2}$,
\begin{eqnarray}\label{acotMA1A3}
|M_{\ell ,s,i,j}(\theta ,\varphi )|&\leq &C\left(1+(\sin |\theta -\varphi |)^j\int_{\frac{1}{2}}^1\frac{(1-r)^k}{\Delta _r^{\lambda +s+1}}dr\right)\nonumber\\
&\leq &C\left(1+(\sin |\theta -\varphi |)^j\int_{\frac{1}{2}}^1\frac{(1-r)^{2s-j}}{(\Delta +(1-r)^2)^{\lambda +s+1}}dr\right)\nonumber\\
&\leq&C\left(1+\frac{(\sin |\theta -\varphi |)^j}{\Delta ^{\lambda +\frac{j+1}{2}}}\int_0^{\frac{1}{2\sqrt{\Delta }}}\frac{u^{2s-j}}{(1+u^2)^{\lambda +s+1}}du\right)\nonumber\\
&\leq &C\frac{1}{(\sin |\theta -\varphi |)^{2\lambda +1}}\\
&\leq& C\left\{\begin{array}{ll}
                                                    (\sin \varphi )^{-2\lambda -1},&\,\,\varphi \geq \frac{3\theta}{2},\\[0.2cm]
                                                    (\sin \theta )^{-2\lambda -1},& \,\,\varphi \leq \frac{\theta }{2}.\nonumber
                                                    \end{array}
\right.
\end{eqnarray}

Suppose now that $\frac{\theta}{2} \le \varphi \le
\frac{3\theta}{2}$, $\theta\neq\varphi$. One can write
\begin{multline*}
\int_{\frac{1}{2}}^1\int_0^\pi(1-r)^k\frac{|b|^j(\sin t)^{2\lambda-1}}{D_r^{\lambda +1+s}}dtdr\\
\shoveleft{\hspace{40mm}\leq C\left(\int_{\frac{1}{2}}^1\int_0^\pi(1-r)^k\frac{(\sin|\theta-\varphi|)^j(\sin t)^{2\lambda-1}}{D_r^{\lambda+1+s}}dtdr \right.}\\
\shoveleft{\hspace{44mm}+\left.\int_{\frac{1}{2}}^1\int_0^{\frac{\pi}{2}}(1-r)^k\frac{(\sin\varphi(1-\cos t))^j(\sin t)^{2\lambda-1}}{D_r^{\lambda+1+s}}dtdr \right.}\\
\left.+\int_{\frac{1}{2}}^1\int_{\frac{\pi}{2}}^\pi(1-r)^k\frac{(\sin \varphi)^j(\sin t)^{2\lambda-1}}{D_r^{\lambda+1+s}} dtdr \right) = \sum^3_{\beta =1}I_{\beta}(\theta,\varphi).\hspace{10mm}
\end{multline*}
We analyze the first integral. By making two changes of variables,
as in \cite[p. 1235]{BMT}, and by taking into account that $2s-j
\leq \ell \leq k-1$ and that $\Delta =2(1-\cos(\theta-\varphi))
\sim (\sin|\theta-\varphi|)^2\sim (\theta -\varphi )^2$ we get
\begin{eqnarray*}
I_1(\theta,\varphi) & \leq & C (\sin|\theta-\varphi|)^j\int_{\frac{1}{2}}^1 (1-r)^k\int_0^\pi
\frac{(\sin t)^{2\lambda-1}}{(\Delta_r+2r\sigma(1-\cos t))^{\lambda +1+s}}dtdr\\
& \leq & C (\sin|\theta-\varphi|)^j\int_{\frac{1}{2}}^1 (1-r)^k\int_0^{\frac{\pi}{2}}
\frac{(\sin t)^{2\lambda-1}}{(\Delta_r+2r\sigma(1-\cos t))^{\lambda +1+s}}dtdr\\
& \leq &C (\sin|\theta-\varphi|)^j \int_{\frac{1}{2}}^1 (1-r)^k\int_0^{\frac{\pi}{2}}\frac{t^{2\lambda-1}}{(\Delta_r+\sigma t^2)^{\lambda +1+s}}dtdr \\
&\leq& C(\sin|\theta-\varphi|)^j\int_{\frac{1}{2}}^1 \frac{(1-r)^k}{\Delta_r^{\lambda +s+1}}
\Big(\sqrt{\frac{\Delta_r}{\sigma}}\Big)^{2\lambda}\int_0^{\frac{\pi}{2}
\sqrt{\frac{\sigma}{\Delta_r}}}\frac{u^{2\lambda -1}}{(1+u^2)^{\lambda +s+1}}dudr \\
&\leq& C \frac{(\sin|\theta-\varphi|)^j}{\sigma^\lambda}\int_{\frac{1}{2}}^1\frac{(1-r)^k}{\Delta_r^{1+s}}dr \\
&\leq& C \frac{(\sin|\theta-\varphi|)^j}{\sigma^\lambda}\int_{\frac{1}{2}}^1\frac{(1-r)^{2s-j+1/2}}{(\Delta +(1-r)^2)^{1+s}}dr\\
& \leq
&C\frac{(\sin|\theta-\varphi|)^j(\sqrt{\Delta})^{2s-j+3/2}}{\sigma^\lambda\Delta^{1+s}}
\int_0^{\frac{1}{2\sqrt{\Delta}}} \frac{u^{2s-j+1/2}}{(1+u^2)^{1+s}}du \\
&\leq&C\frac{(\sin|\theta-\varphi|)^j}{\sigma^\lambda\Delta^{\frac{1}{4}+\frac{j}{2}}}\int_0^\infty
\frac{u^{2s-j+1/2}}{(1+u^2)^{1+s}}du \\
&\leq& C\frac{1}{\sigma ^\lambda\sqrt{|\theta-\varphi|}}.
\end{eqnarray*}

For the second integral we write
\begin{eqnarray*}
I_2(\theta,\varphi) &\leq & C(\sin\varphi)^j\int_{\frac{1}{2}}^1(1-r)^k\int_0^{\frac{\pi}{2}}\frac{t^{2\lambda+2j-1}}{(\Delta_r+\sigma t^2)^{\lambda+1+s}}dtdr\\
&\le&C(\sin\varphi)^j\int_{\frac{1}{2}}^1(1-r)^k\int_0^{\frac{\pi}{2}}\frac{t^{2\lambda+j-1}}{(\Delta_r+\sigma t^2)^{\lambda+1+s}}dtdr\\
&\leq& C\frac{(\sin\varphi)^j}{\sigma^{\lambda +\frac{j}{2}}}\int_{\frac{1}{2}}^1\frac{(1-r)^k}{\Delta_r^{1+s-\frac{j}{2}}}dr\\
&\leq& C\frac{(\sin \varphi )^j}{\sigma ^{\lambda +\frac{j}{2}}}
\int_{\frac{1}{2}}^1\frac{(1-r)^{2s-j+\frac{1}{2}}}{(\Delta +(1-r)^2)^{1+s-\frac{j}{2}}}dr \\
& \leq & C\frac{(\sin\varphi)^j(\sqrt{\Delta})^{2s-j+\frac{3}{2}}}{\sigma^{\lambda+\frac{j}{2}}\Delta^{1+s-\frac{j}{2}}}\int_0^{\frac{1}{2\sqrt{\Delta}}}\frac{u^{2s-j+\frac{1}{2}}}
{(1+u^2)^{1+s-\frac{j}{2}}}du\\
&\leq& C \frac{(\sin\varphi)^j}{\sigma^{\lambda+\frac{j}{2}}\Delta^{\frac{1}{4}}} \leq C \frac{1}{\sigma ^\lambda \sqrt{|\theta-\varphi|}},
\end{eqnarray*}
because $\sin\theta \sim \sin\varphi$.
Finally it has
\begin{eqnarray*}
I_3(\theta,\varphi) &\leq & C(\sin\varphi)^j\int_{\frac{1}{2}}^1\int_{\frac{\pi}{2}}^\pi (1-r)^k\frac{(\sin t)^{2\lambda-1}}{(\Delta_r+\sigma)^{\lambda+1+s}}dtdr\\
&\leq& C\frac{(\sin\varphi)^j}{\sigma^{\lambda+\frac{j}{2}}}\int_{\frac{1}{2}}^1 \frac{(1-r)^k}{\Delta_r^{1+s-\frac{j}{2}}}dr\leq C\frac{1}{\sigma ^\lambda\sqrt{|\theta-\varphi|}}.
\end{eqnarray*}
Then we conclude that if $\frac{\theta}{2}\leq \varphi \leq
\frac{3\theta }{2}$, $\theta\neq\varphi$,
$$
|M_{\ell ,s,i,j}(\theta , \varphi)|\leq
C\left(1+\int_{\frac{1}{2}}^1\int_0^\pi(1-r)^k\frac{|b|^j(\sin
t)^{2\lambda-1}}{D_r^{\lambda +1+s}}dtdr\right)\leq C\frac{1}{\sigma
^\lambda\sqrt{|\theta-\varphi|}},
$$
that, jointly with (\ref{acotMA1A3}), gives (\ref{acotM}).

We have proved that the integral in (\ref{lderivada}) is absolutely
convergent. By analyzing carefully the above estimates we can also
see that, for every $\ell=0,1,...,k-2$, $R_\lambda^{k,\ell}$ is a
continuous function on $(0,\pi)\times(0,\pi)$. Then, we conclude
(\ref{lderivada}), for every $\ell=0,1,...,k-1$. Note that when
$\ell=k-1$ to prove (\ref{lderivada}) we need to use distributional
arguments (see Lemma \ref{derint1} in Appendix).

{\bf Step 2.} We now study the kernel
$$
R_\lambda^k(\theta, \varphi)=\frac{\lambda}{\pi\Gamma(k)}\frac{\partial^k}{\partial\theta^k}\left[\int_0^1r^{\lambda-1}\Big(\log\frac{1}{r}\Big)^{k-1}(1-r^2)
\int_0^\pi\frac{(\sin t)^{2\lambda-1}}{D_r^{\lambda+1}}dtdr\right],
$$
for $\theta,\varphi \in (0,\pi)$. We get estimates which are better than the ones obtained in (\ref{acotR}).
\begin{lema}\label{lema}
Let $\lambda >0$ and $k \in \mathbb N$. Then,
$$
R_\lambda^k(\theta,\varphi)= \left\{\begin{array}{ll}
                                    O((\sin\varphi )^{-(2\lambda+1)}), & (\theta ,\varphi )\in A_1;\\
                                    \displaystyle \frac{R^k(\theta ,\varphi )}{(\sin\theta\sin\varphi)^\lambda}+O\left(\frac{1}{(\sin \varphi )^{2\lambda +1}}\left(1+\sqrt{\frac{\sin \varphi }{|\theta -\varphi |}}\right)\right),&(\theta ,\varphi )\in A_2, \,\,\theta \neq\varphi;\\
                                    O((\sin\theta )^{-(2\lambda+1)}),&  (\theta ,\varphi )\in
                                    A_3;
                                    \end{array}\right.
$$
where
$$
R^k(\theta,\varphi) = \frac{1}{2\pi \Gamma
(k)}\frac{\partial^k}{\partial\theta^k}\int_0^1\left(\log\frac{1}{r}\right)^{k-1}\left(\frac{1-r^2}{1-2r\cos(\theta-\varphi)+r^2}-1\right)\frac{dr}{r}.
$$
\end{lema}

\begin{proof} Since $R_\lambda ^k(\theta, \varphi )=(-1)^kR_\lambda ^k(\pi -\theta ,\pi -\varphi)$ and $R^k(\theta ,\varphi )=(-1)^kR^k(\pi -\theta ,\pi -\varphi )$,
 $\theta, \varphi \in (0,\pi )$, we can assume $(\theta ,\varphi )\in (0,\frac{\pi }{2})\times (0,\pi )$.

When $( \theta ,\varphi )\in A_1\cup A_3$, we can argue as in the proof of (\ref{acotMA1A3}), for $\ell =k$, and thus we get
$$
|R_\lambda^k(\theta,\varphi)|=\Big|\frac{\lambda}{\pi\Gamma(k)} \sum_{\begin{subarray}{c}
                                                            s=1,\ldots,k\\j \geq 2s-k\\i+j=s
                                                            \end{subarray}} c_{k,s,i,j}M_{k,s,i,j}(\theta ,\varphi )\Big|\leq C\left\{\begin{array}{ll}
                                                    (\sin \varphi )^{-2\lambda -1},&\,\,\varphi \geq \frac{3\theta}{2},\\[0.2cm]
                                                    (\sin \theta )^{-2\lambda -1},&\,\,\varphi \leq \frac{\theta }{2}.
                                                    \end{array}
\right.
$$

We now consider $\frac{\theta}{2} \leq \varphi \leq
\frac{3\theta}{2}$ and $\varphi \neq \theta$. First we write
\begin{eqnarray}\label{R}
R_\lambda^k(\theta,\varphi)&
\hspace{-2cm}=&\hspace{-2cm}\frac{\lambda}{\pi\Gamma(k)}\left[\int_0^1\int_{\frac{\pi
}{2}}^\pi+\int_0^{1-\frac{\sqrt{\sigma }}{2}}\int_0^{\frac{\pi
}{2}}+\int_{1-\frac{\sqrt{\sigma }}{2}}^1\int_0^{\frac{\pi
}{2}}\right]r^{\lambda
-1}\left(\log\frac{1}{r}\right)^{k-1}(1-r^2)\nonumber\\
&\hspace{2cm}\times&\frac{\partial ^k}{\partial \theta
^k}\frac{(\sin
t)^{2\lambda -1}}{D_r^{\lambda +1}}dtdr\nonumber\\
&\hspace{-2cm}=&\hspace{-2cm}\frac{\lambda}{\pi\Gamma(k)}[I_\lambda
^{k,1}(\theta,\varphi)+I_\lambda ^{k,2}(\theta,\varphi)+I_\lambda
^{k,3}(\theta,\varphi)].
\end{eqnarray}

Let us decompose $I_\lambda ^{k,3}(\theta ,\varphi )$ as follows,
\begin{eqnarray}\label{I3}
I_\lambda ^{k,3}(\theta,\varphi)&=&
\int_{1-\frac{\sqrt{\sigma}}{2}}^1r^{\lambda-1}\left(\log\frac{1}{r}\right)^{k-1}
(1-r^2)\frac{\partial^k}{\partial\theta^k}\int_0^{\frac{\pi}{2}}\left(\frac{(\sin
t)^{2\lambda-1}}
{D_r^{\lambda+1}}-\frac{t^{2\lambda -1}}{(\Delta_r+r\sigma t^2)^{\lambda+1}}\right)dtdr\nonumber\\
&+&\int_{1-\frac{\sqrt{\sigma}}{2}}^1r^{\lambda-1}\left(\log\frac{1}{r}\right)^{k-1}
                                             (1-r^2)\frac{\partial^k}{\partial\theta^k}\int_0^{\frac{\pi}{2}}\frac{t^{2\lambda-1}}
                                             {(\Delta_r+r\sigma
                                             t^2)^{\lambda+1}}dtdr\nonumber\\
&=&J_\lambda ^k(\theta ,\varphi )+K_\lambda ^k(\theta ,\varphi ).
\end{eqnarray}

Moreover, we observe that by making the change of variable
$u=\sqrt{\frac{r\sigma}{\Delta_r}}t$,
\begin{eqnarray*}
\int_0^{\frac{\pi}{2}}\frac{t^{2\lambda-1}}{(\Delta_r+r\sigma
t^2)^{\lambda+1}}dt &=&\left(\int_0^\infty
-\int_{\frac{\pi}{2}}^\infty
\right)\frac{t^{2\lambda-1}}{(\Delta_r+r\sigma
t^2)^{\lambda+1}}dt\\
&=&
\frac{1}{(r\sigma)^\lambda\Delta_r}\int_0^\infty\frac{u^{2\lambda-1}}{(1+u^2)^{\lambda+1}}du
-\int_{\frac{\pi}{2}}^\infty\frac{t^{2\lambda-1}}{(\Delta_r+r\sigma
t^2)^{\lambda+1}}dt\\
&=& \frac{1}{2\lambda
(r\sigma)^\lambda\Delta_r}-\int_{\frac{\pi}{2}}^\infty\frac{t^{2\lambda-1}}{(\Delta_r+r\sigma
t^2)^{\lambda+1}}dt,\quad r\in (0,1).
\end{eqnarray*}

Then, Leibniz's rule leads to
\begin{eqnarray*}
K_\lambda ^k(\theta ,\varphi )&=&\frac{1}{2\lambda
}\left[\int_{1-\frac{\sqrt{\sigma}}{2}}^1\left(\log\frac{1}{r}\right)^{k-1}
                                             \frac{\partial^k}{\partial\theta^k}\Big(\frac{1}{\sigma ^\lambda
}\frac{(1-r^2)}
                                             {\Delta_r}\Big)\frac{dr}{r}\right]\\
&-&\int_{1-\frac{\sqrt{\sigma}}{2}}^1
r^{\lambda-1}\left(\log\frac{1}{r}\right)^{k-1}(1-r^2)\frac{\partial^k}{\partial\theta^k}\int_{\frac{\pi}{2}}^\infty\frac{t^{2\lambda-1}}{(\Delta_r+r\sigma
t^2)^{\lambda+1}}dtdr\\
&=&\frac{1}{2\lambda \sigma ^\lambda
}\left[\int_{1-\frac{\sqrt{\sigma}}{2}}^1\left(\log\frac{1}{r}\right)^{k-1}
                                             \frac{\partial^k}{\partial\theta^k}\frac{(1-r^2)}
                                             {\Delta_r}\frac{dr}{r}\right]\\
&+&\frac{1}{2\lambda }\sum_{n=0}^{k-1}\binom{k}{n}
\frac{\partial^{k-n}}{\partial\theta^{k-n}}\left(\frac{1}{\sigma^\lambda}\right)
\left[\int_{1-\frac{\sqrt{\sigma}}{2}}^1\left(\log\frac{1}{r}\right)^{k-1}
                                             \frac{\partial^n}{\partial\theta^n}\frac{(1-r^2)}
                                             {\Delta_r}\frac{dr}{r}\right]\\
&-&\int_{1-\frac{\sqrt{\sigma}}{2}}^1
r^{\lambda-1}\left(\log\frac{1}{r}\right)^{k-1}(1-r^2)\frac{\partial^k}{\partial\theta^k}\int_{\frac{\pi}{2}}^\infty\frac{t^{2\lambda-1}}{(\Delta_r+r\sigma
t^2)^{\lambda+1}}dtdr.
\end{eqnarray*}

We observe that
\begin{equation}\label{K}
K_\lambda ^k(\theta ,\varphi )=\frac{\pi \Gamma (k)}{\lambda \sigma
^\lambda }R^k(\theta, \varphi )+\sum_{\beta=1}^3K_\lambda ^{k,\beta
}(\theta ,\varphi ),
\end{equation}
where
$$
K_\lambda ^{k,1}(\theta ,\varphi )= -\frac{1}{2\lambda \sigma
^\lambda }
\left[\int_0^{1-\frac{\sqrt{\sigma}}{2}}\left(\log\frac{1}{r}\right)^{k-1}
\frac{\partial^k}{\partial\theta^k}\frac{(1-r^2)}{\Delta_r}\frac{dr}{r}\right],\hspace{3.3cm}
$$
$$
K_\lambda ^{k,2}(\theta ,\varphi )=\frac{1}{2\lambda
}\sum_{n=0}^{k-1}\binom{k}{n}
\frac{\partial^{k-n}}{\partial\theta^{k-n}}\left(\frac{1}{\sigma^\lambda}\right)
\left[\int_{1-\frac{\sqrt{\sigma}}{2}}^1\left(\log\frac{1}{r}\right)^{k-1}
                                             \frac{\partial^n}{\partial\theta^n}\frac{(1-r^2)}
                                             {\Delta_r}\frac{dr}{r}\right],
$$
and
$$
K_\lambda ^{k,3}(\theta ,\varphi
)=-\int_{1-\frac{\sqrt{\sigma}}{2}}^1
r^{\lambda-1}\left(\log\frac{1}{r}\right)^{k-1}(1-r^2)\frac{\partial^k}{\partial\theta^k}\int_{\frac{\pi}{2}}^\infty\frac{t^{2\lambda-1}}{(\Delta_r+r\sigma
t^2)^{\lambda+1}}dtdr.
$$

Thus, according to (\ref{R}), (\ref{I3}) and (\ref{K}), to establish
our result we must analyze $I_\lambda ^{k,\beta }(\theta ,\varphi
)$, $\beta =1,2$, $J_\lambda ^k(\theta ,\varphi )$ and $K_\lambda
^{k,\beta}(\theta ,\varphi )$, $\beta =1,2,3$.

Let us consider first $I_\lambda ^{k,1}(\theta ,\varphi )$. We
will see that
\begin{equation}\label{acotI1}
|I_\lambda ^{k,1}(\theta ,\varphi )|\leq \frac{C}{(\sin \varphi
)^{2\lambda +1}}\left(1+\sqrt{\frac{\sin \varphi }{|\theta
-\varphi |}}\right).
\end{equation}

Let $s=1,\ldots,k$, $j \geq 2s-k$ and $i+j=s$. We define
$$
I_{\lambda,s,i,j}^{k,1}(\theta ,\varphi )=\int_0^1 r^{i+j+\lambda
-1}\left(\log\frac{1}{r}\right)^{k-1}(1-r^2)\int_{\frac{\pi
}{2}}^\pi\frac{a^ib^j(\sin t)^{2\lambda -1}}{D_r^{\lambda
+s+1}}dtdr.
$$
According to (\ref{derivDr}) it is then sufficient to obtain
(\ref{acotI1}) when $I_{\lambda,s,i,j}^{k,1}(\theta ,\varphi )$
replaces to $I_\lambda^{k,1}(\theta,\varphi)$. By proceeding as in
Step 1, using (\ref{bj}), since $D_r\geq C$, for $0<r<\frac{1}{2}$,
and $D_r\geq (\Delta _r+\sigma )$, for $\frac{1}{2}<r<1$ and $t\in
(\frac{\pi }{2},\pi )$, it has
\begin{eqnarray*}
|I_{\lambda,s,i,j}^{k,1}(\theta,\varphi)|&\leq& C\int_0^1 r^{\lambda
-1}\left(\log\frac{1}{r}\right)^{k-1}(1-r^2)\int_{\frac{\pi}{2}}^\pi
\frac{|b|^j(\sin\,t)^{2\lambda-1}}{D_r^{\lambda +s+1}}dtdr\\
&\leq&C\left(1+(\sin \varphi)^j\int_{\frac{1}{2}}^1
\frac{(1-r)^k}{(\Delta_r+\sigma )^{\lambda +s+1}}dr \right)\\
&\leq&C\left(1+\frac{(\sin\varphi)^j}{\sigma^{\lambda+\frac{1}{4}+\frac{j}{2}}}\int_{\frac{1}{2}}^1
\frac{(1-r)^{2s-j}}{\Delta_r^{s+\frac{3}{4}-\frac{j}{2}}}dr \right)\\
&\leq&C\left(1+\frac{1}{\sigma^{\lambda+\frac{1}{4}}\Delta^{\frac{1}{4}}}
\int_0^\infty \frac{u^{2s-j}}{(1+u^2)^{s+\frac{3}{4}-\frac{j}{2}}}du\right)\\
&\leq& C\left(1+\frac{1}{(\sin\varphi )
^{2\lambda+\frac{1}{2}}\sqrt{|\theta-\varphi|}}\right) \leq
\frac{C}{(\sin\varphi )^{2\lambda+1}}\left(1+\sqrt{\frac{\sin
\varphi }{|\theta-\varphi|}}\right).
\end{eqnarray*}

For $I_\lambda ^{k,2}(\theta ,\varphi )$ we proceed in a similar
way. Consider $s=1,\ldots,k$, $j \geq 2s-k$, and $i+j=s$, and
$$
I_{\lambda,s,i,j}^{k,2}(\theta ,\varphi
)=\int_0^{1-\frac{\sqrt{\sigma }}{2}}r^{i+j+\lambda
-1}\left(\log\frac{1}{r}\right)^{k-1}(1-r^2)\int_0^{\frac{\pi
}{2}}\frac{a^ib^j(\sin t)^{2\lambda -1}}{D_r^{\lambda +s+1}}dtdr.
$$

We have that
\begin{eqnarray*}
|I_{\lambda,s,i,j}^{k,2}(\theta ,\varphi )|&\leq&
C\left(\int_0^{\frac{1}{2}}r^{\lambda-1}\left(\log\frac{1}{r}\right)^{k-1}dr+(\sin
\varphi
)^j\int_{\frac{1}{2}}^{1-\frac{\sqrt{\sigma}}{2}}\frac{(1-r)^k}{\Delta
_r^{\lambda +s+1}}dr\right)\\
&\leq&
C\left(1+(\sin\varphi)^j\int_{\frac{1}{2}}^{1-\frac{\sqrt{\sigma}}{2}}
\frac{(1-r)^{2s-j}}{(1-r)^{2\lambda+2s+2}}dr\right) \leq
\frac{C}{(\sin\varphi )^{2\lambda+1}}.
\end{eqnarray*}

Hence,
\begin{equation}\label{acotI2}
|I_\lambda ^{k,2}(\theta ,\varphi )|\leq \frac{C}{(\sin \varphi
)^{2\lambda +1}}.
\end{equation}

To estimate $J_\lambda ^k(\theta ,\varphi )$ (see (\ref{I3})), we write
$$
J_\lambda ^k(\theta ,\varphi )=J_\lambda ^{k,1}(\theta ,\varphi
)+J_\lambda ^{k,2}(\theta ,\varphi ),
$$
where
$$
J_\lambda ^{k,1}(\theta ,\varphi
)=\int_{1-\frac{\sqrt{\sigma}}{2}}^1r^{\lambda -1
}\left(\log\frac{1}{r}\right)^{k-1}(1-r^2)\frac{\partial
^k}{\partial \theta ^k}\int_0^{\frac{\pi }{2}}\frac{(\sin
t)^{2\lambda -1}-t^{2\lambda -1}}{D_r^{\lambda +1}}dtdr,
$$
and
$$
J_\lambda ^{k,2}(\theta ,\varphi
)=\int_{1-\frac{\sqrt{\sigma}}{2}}^1r^{\lambda -1
}\left(\log\frac{1}{r}\right)^{k-1}(1-r^2)\frac{\partial
^k}{\partial \theta ^k}\int_0^{\frac{\pi }{2}}t^{2\lambda
-1}\left(\frac{1}{D_r^{\lambda +1}}-\frac{1}{(\Delta _r+r\sigma
t^2)^{\lambda+1}}\right)dtdr.
$$

To analyze $J_\lambda ^{k,1}(\theta ,\varphi )$ assume, as above,
$s=1,...,k$, $j\geq 2s-k$, and $i+j=s$  and consider
$$
J_{\lambda,s,i,j}^{k,1}(\theta ,\varphi
)=\int_{1-\frac{\sqrt{\sigma}}{2}}^1r^{i+j+\lambda-1
}\left(\log\frac{1}{r}\right)^{k-1}(1-r^2)\int_0^{\frac{\pi
}{2}}\frac{a^ib^j[(\sin t)^{2\lambda -1}-t^{2\lambda
-1}]}{D_r^{\lambda +s+1}}dtdr.
$$

By using the mean value theorem and that $|b|^j\leq C(|\theta
-\varphi |^j+(t^2\sin \varphi )^j)$, $t\in (0,\frac{\pi }{2})$, we
have
\begin{eqnarray}\label{Jksij1}
|J_{\lambda,s,i,j}^{k,1}(\theta ,\varphi
)|&\leq&C\int_{1-\frac{\sqrt{\sigma}}{2}}^1(1-r)^k\int_0^{\frac{\pi
}{2}}\frac{|b|^jt^{2\lambda +1}}{D_r^{\lambda +s+1}}dtdr\\
&\leq&C\int_{1-\frac{\sqrt{\sigma}}{2}}^1(1-r)^k\int_0^{\frac{\pi
}{2}}\frac{[|\theta
-\varphi |^j+(t^2\sin \varphi )^j]t^{2\lambda +1}}{(\Delta _r+\sigma t^2)^{\lambda +s+1}}dtdr\nonumber\\
&=&C(J_{\lambda,s,i,j}^{k,1,1}(\theta ,\varphi
)+J_{\lambda,s,i,j}^{k,1,2}(\theta ,\varphi )).\nonumber
\end{eqnarray}

We can obtain for each term in the last sum the following estimates.
Firstly,
\begin{eqnarray*}
|J_{\lambda,s,i,j}^{k,1,1}(\theta ,\varphi )|&\leq& C\frac{|\theta -
\varphi |^j}{\sigma ^{\lambda
+1}}\int_{1-\frac{\sqrt{\sigma}}{2}}^1\frac{(1-r)^k}{\Delta
_r^s}\int_0^{\frac{\pi }{2}\sqrt{\frac{\sigma }{\Delta
_r}}}\frac{u^{2\lambda +1}}{(1+u^2)^{\lambda +s+1}}dudr\\
&\leq&C\frac{|\theta - \varphi |^j}{\sigma ^{\lambda
+1}}\int_{1-\frac{\sqrt{\sigma}}{2}}^1\frac{(1-r)^k}{(\Delta
+(1-r)^2)^{s-\frac{j}{2}+\frac{j}{2}}}dr\\
&\leq&C\frac{1}{\sigma ^{\lambda
+1}}\int_{1-\frac{\sqrt{\sigma}}{2}}^1(1-r)^{k-2s+j}dr\leq
C\frac{\sigma^{\frac{k+j}{2}-s}}{\sigma^{\lambda +\frac{1}{2}}} \leq
\frac{C}{(\sin \varphi )^{2\lambda +1}}.
\end{eqnarray*}

We also have
\begin{eqnarray*}
|J_{\lambda,s,i,j}^{k,1,2}(\theta ,\varphi )|&\leq& C(\sin \varphi
)^j\int_{1-\frac{\sqrt{\sigma}}{2}}^1(1-r)^k\int_0^{\frac{\pi
}{2}}\frac{t^{2\lambda +2j+1}}{((1-r)^2+\sigma t^2)^{\lambda
+s+1}}dtdr\\
&\leq&C\frac{(\sin \varphi)^j}{\sigma ^{\lambda
+\frac{j}{2}+\frac{3}{4}}}\int_{1-\frac{\sqrt{\sigma}}{2}}^1(1-r)^{k-2s+j-\frac{1}{2}}dr\int_0^{\frac{\pi
}{2}}t^{j-\frac{1}{2}}dt\\
&\leq& C\frac{\sigma^{\frac{k+j}{2}-s}}{\sigma^{\lambda
+\frac{1}{2}}} \leq \frac{C}{(\sin \varphi )^{2\lambda +1}}.
\end{eqnarray*}

Thus we get by (\ref{derivDr})
\begin{equation}\label{J1}
|J_\lambda ^{k,1}(\theta ,\varphi )|\leq
 \frac{C}{(\sin \varphi )^{2\lambda +1}}.
\end{equation}

In the same way, and according to (\ref{derivDr}), to see the
estimate for $J_\lambda ^{k,2}(\theta ,\varphi )$ we analyze
\begin{eqnarray*}
J_{\lambda,s,i,j}^{k,2}(\theta ,\varphi
)&=&\int_{1-\frac{\sqrt{\sigma}}{2}}^1r^{\lambda +i+j-1}\left(\log
\frac{1}{r}\right)^{k-1}(1-r^2)\\
&\times&\int_0^{\frac{\pi }{2}} t^{2\lambda
-1}\left(\frac{a^ib^j}{D_r^{\lambda +s+1}}-\frac{A^iB^j}{(\Delta
_r+r\sigma t^2)^{\lambda +s+1}}\right)dtdr,
\end{eqnarray*}
for every $s=1,...,k$, $j\geq 2s-k$, and $i+j=s$. Here $A=\cos
(\theta -\varphi )-\frac{\sigma t^2}{2}$ and $B=\frac{\partial
A}{\partial \theta }$.

By using the mean value theorem we obtain, for $t\in (0,\frac{\pi
}{2})$,
\begin{eqnarray*}
\left|\frac{a^ib^j}{D_r^{\lambda +s+1}}-\frac{A^iB^j}{(\Delta
_r+r\sigma t^2)^{\lambda +s+1}}\right|\\
&\hspace{-8cm}\leq&\hspace{-4cm}
\left|a^ib^j\left(\frac{1}{D_r^{\lambda +s+1}}-\frac{1}{(\Delta
_r+r\sigma t^2)^{\lambda
+s+1}}\right)\right|+\left|\frac{(a^i-A^i)b^j+A^i(b^j-B^j)}{(\Delta
_r+r\sigma
t^2)^{\lambda +s+1}}\right|\\
&\hspace{-8cm}\leq &\hspace{-4cm}C\left(\frac{|b|^j\sigma
t^4}{(\Delta _r+r\sigma t^2)^{\lambda
+s+2}}+\frac{|b|^{j-1}\sqrt{\sigma} t^4}{(\Delta _r+r\sigma
t^2)^{\lambda
+s+1}}\right)\\
&\hspace{-8cm}\leq &\hspace{-4cm}C\left(\frac{|b|^jt^2}{(\Delta
_r+\sigma t^2)^{\lambda +s+1}}+ \frac{|b|^{j-1}\sqrt{\sigma}
t^3}{(\Delta _r+\sigma t^2)^{\lambda +s+1}}\right),
\end{eqnarray*}
where the second term in the two last sums does not appear when
$j=0$. Then, we write
$$
|J_{\lambda,s,i,j}^{k,2}(\theta ,\varphi )|\leq
C(J_{\lambda,s,i,j}^{k,2,1}(\theta ,\varphi
)+J_{\lambda,s,i,j}^{k,2,2}(\theta ,\varphi )),
$$
where
$$
J_{\lambda,s,i,j}^{k,2,1}(\theta ,\varphi
)=\int_{1-\frac{\sqrt{\sigma}}{2}}^1(1-r)^k\int_0^{\frac{\pi
}{2}}\frac{|b|^jt^{2\lambda +1}}{(\Delta _r+\sigma t^2)^{\lambda
+s+1}}dtdr,
$$
and
$$
J_{\lambda,s,i,j}^{k,2,2}(\theta ,\varphi )=\sqrt{\sigma}
\int_{1-\frac{\sqrt{\sigma}}{2}}^1(1-r)^k\int_0^{\frac{\pi
}{2}}\frac{|b|^{j-1}t^{2\lambda +2}}{(\Delta _r+\sigma t^2)^{\lambda
+s+1}}dtdr,\quad \mbox{ when }j\geq 1.
$$

We observe that $J_{\lambda,s,i,j}^{k,2,1}(\theta ,\varphi )$ was
already analyzed in (\ref{Jksij1}). On the other hand, when $j\geq
1$, we can use that $|b|^{j-1}\leq C(|\theta -\varphi
|^{j-1}+(t^2\sin \varphi )^{j-1})$, $t\in (0,\frac{\pi }{2})$, and
proceed as in the estimation of (\ref{Jksij1}) to study
$J_{\lambda,s,i,j}^{k,2,2}$. Thus we get that
\begin{equation}\label{J2}
|J_\lambda ^{k,2}(\theta ,\varphi )|\leq
\frac{C}{(\sin \varphi )^{2\lambda +1}}.
\end{equation}

By combining (\ref{J1}) and (\ref{J2}) we conclude that
\begin{equation}\label{acotJk}
|J_\lambda ^k(\theta ,\varphi )|\leq \frac{C}{(\sin \varphi
)^{2\lambda +1}}.
\end{equation}

Finally we deal with $K_\lambda ^{k,\beta }(\theta ,\varphi )$,
$\beta =1,2,3$ (see (\ref{K})).

By invoking (\ref{derivDr}) with $t=0$ and $\lambda=0$, it has that
$$
\frac{\partial ^{\ell }}{\partial \theta ^{\ell
}}\left(\frac{1}{\Delta _r}\right)=\sum_{s,i,j} c_{\ell
,s,i,j}\frac{r^{i+j}(\cos (\theta -\varphi))^i(-\sin (\theta
-\varphi ))^j}{\Delta _r^{1+s}} ,\quad \ell \in \mathbb{N},
$$
where $c_{\ell ,s,i,j}\not =0$ only if $s=1,...,\ell $, $j\geq
2s-\ell$ and $i+j=s$.

Also, for every $n=0,...,k-1$,
\begin{equation}\label{derivsigma}
\left|\frac{\partial ^{k-n}}{\partial \theta
^{k-n}}\Big(\frac{1}{\sigma ^\lambda}\Big)\right|\leq C(\sin \theta
)^{-\lambda -k+n}(\sin \varphi )^{-\lambda }\leq C\sigma ^{-\lambda
-\frac{k-n}{2}}.
\end{equation}

Hence, by proceeding as above, to estimate $K_\lambda ^{k,\beta
}(\theta ,\varphi )$, $\beta =1,2,3$, it is sufficient to study the
following integrals:
$$
K^{k,1}_{\lambda,s,i,j}(\theta ,\varphi
)=\frac{1}{\sigma^\lambda}\int_0^{1-\frac{\sqrt{\sigma
}}{2}}r^{i+j-1}\left(\log \frac{1}{r}\right)^{k-1}\frac{(1-r^2)(\cos
(\theta -\varphi ))^i(-\sin (\theta -\varphi ))^j} {\Delta
_r^{s+1}}dr,
$$
when $s=1,...,k$, $j\geq 2s-k$ and  $i+j=s$;
$$
K^{k,2,0}_\lambda(\theta,\varphi)=\sigma^{-k/2-\lambda}\int_{1-\frac{\sqrt{\sigma}}{2}}^1\left(\log\frac{1}{r}\right)^{k-1}
\frac{(1-r^2)}{\Delta_r}\frac{dr}{r};
$$
$$
K^{k,2,n}_{\lambda,s,i,j}(\theta ,\varphi
)=\frac{\partial^{k-n}}{\partial \theta ^{k-n}}\Big(\frac{1}{\sigma
^\lambda }\Big)\int_{1-\frac{\sqrt{\sigma
}}{2}}^1r^{i+j-1}\left(\log
\frac{1}{r}\right)^{k-1}\frac{(1-r^2)(\cos (\theta -\varphi
))^i(-\sin (\theta -\varphi ))^j} {\Delta _r^{s+1}}dr,
$$
for each $n=1,...,k-1$, $s=1,...,n$, $j\geq 2s-n$, and $i+j=s$; and
$$
K^{k,3}_{\lambda,s,i,j}(\theta ,\varphi )=\int_{1-\frac{\sqrt{\sigma
}}{2}}^1r^{i+j+\lambda
-1}\left(\log\frac{1}{r}\right)^{k-1}(1-r^2)\int_{\frac{\pi
}{2}}^\infty \frac{A^iB^jt^{2\lambda -1}}{(\Delta _r+r\sigma
t^2)^{\lambda +s+1}}dtdr,
$$
when $s=1,...,k$, $j\geq 2s-k$, and $i+j=s$. Here, as before,
$A=\cos (\theta -\varphi )-\frac{\sigma t^2}{2}$ and
$B=\frac{\partial A}{\partial \theta }$.

Consider $s=1,...,k$, $j\geq 2s-k$, and $i+j=s$. We can write
\begin{eqnarray*}
|K^{k,1}_{\lambda,s,i,j}(\theta ,\varphi )|&\leq
&C\frac{1}{\sigma^\lambda}\left(\int_0^{\frac{1}{2}}r^{i+j-1}
\left(\log
\frac{1}{r}\right)^{k-1}dr+\int_{\frac{1}{2}}^{1-\frac{\sqrt{\sigma }}{2}}\frac{(1-r)^k|\theta -\varphi |^j}{(\Delta +(1-r)^2)^{s+1}}dr\right)\\
&\leq &C\frac{1}{\sigma^\lambda}\left(1+\frac{|\theta -\varphi |^j}{\Delta ^{\frac{j}{2}}}\int_{\frac{1}{2}}^{1-\frac{\sqrt{\sigma }}{2}}(1-r)^{k-2s+j-2}dr\right)\\
&\leq&C\frac{1}{\sigma^\lambda}\left(1+\frac{\sigma
^{\frac{k+j}{2}-s}}{\sqrt{\sigma }}\right)\le \frac{C}{(\sin \varphi
)^{2\lambda +1}}.
\end{eqnarray*}

On the other hand, by using (\ref{derivsigma}), for each
$n=1,...,k-1$, $s=1,...,n$, $j\geq 2s-n$, and $i+j=s$, we obtain
\begin{eqnarray*}
|K^{k,2,n}_{\lambda,s,i,j}(\theta ,\varphi )|&\leq &C\frac{|\theta
-\varphi |^j}{\sigma ^{\lambda +\frac{k-n}{2}}}
\int_{1-\frac{\sqrt{\sigma}}{2}}^1\frac{(1-r)^k}{(\Delta +(1-r)^2)^{s+1}}dr\\
&\leq&C\frac{|\theta -\varphi |^j}{\sigma ^{\lambda +\frac{k-n}{2}}\Delta ^{\frac{j}{2}+\frac{1}{4}}}
\int_{1-\frac{\sqrt{\sigma}}{2}}^1(1-r)^{k-2s+j-\frac{3}{2}}dr\\
&\leq&C\frac{\sigma ^{\frac{n+j}{2}-s}}{\sigma ^{\lambda +\frac{1}{4}}\Delta ^{\frac{1}{4}}}\leq \frac{C}{(\sin \varphi )^{2\lambda +1}}
\sqrt{\frac{\sin \varphi }{|\theta -\varphi |}}.
\end{eqnarray*}

In a similar way we obtain
$$
|K^{k,2,0}_\lambda(\theta,\varphi)|\le \frac{C}{(\sin \varphi
)^{2\lambda +1}} \sqrt{\frac{\sin \varphi }{|\theta -\varphi |}}.
$$

Finally, assume $s=1,\ldots,k$, $j\geq 2s-k$, and $i+j=s$. By taking
into account that $|B|^j\leq C(t^{2}\sin \varphi )^j$ and $|A|^i\le
C(1+\sigma^i t^{2i})$, $t\geq \frac{\pi }{2}$, and the last formula
in \cite[p. 37]{MS}, we obtain
\begin{eqnarray*}
K^{k,3}_{\lambda,s,i,j}(\theta ,\varphi
)&\leq&C\int^1_{1-\frac{\sqrt{\sigma}}{2}}(1-r)^k\int_{\frac{\pi}{2}}^\infty
\frac{|A|^i|B|^jt^{2\lambda-1}}
{(\Delta_r+\sigma t^2)^{\lambda+s+1}}dtdr\\
&\leq &C\frac{(\sin \varphi )^j}{\sigma ^{\lambda +j}}\Big(
\int_{1-\frac{\sqrt{\sigma }}{2}}^1\frac{(1-r)^k}{\Delta _r^{s+1-j}}
\int_{\frac{\pi }{2}\sqrt{\frac{\sigma }{\Delta
_r}}}^\infty\frac{u^{2\lambda +2j-1}}{(1+u^2)^{\lambda
+s+1}}dudr\\
&+&\int_{1-\frac{\sqrt{\sigma }}{2}}^1\frac{(1-r)^k}{\Delta
_r}\int_{\frac{\pi }{2}\sqrt{\frac{\sigma }{\Delta
_r}}}^\infty\frac{u^{2\lambda +2j+2i-1}}{(1+u^2)^{\lambda
+s+1}}dudr\Big)\\
&\leq &C\frac{(\sin \varphi )^j}{\sigma ^{\lambda +j+1}}
\Big(\int_{1-\frac{\sqrt{\sigma }}{2}}^1\frac{(1-r)^k}{(\Delta
+(1-r)^2)^{s-j}}dr+\int_{1-\frac{\sqrt{\sigma }}{2}}^1(1-r)^kdr\Big)\\
&\leq&C\frac{(\sin \varphi )^j}{\sigma ^{\lambda +j+1}}
\int_{1-\frac{\sqrt{\sigma }}{2}}^1(1-r)^{k-2s+2j}dr \leq
C\frac{(\sin \varphi )^j\sigma^{\frac{k+j}{2}-s}}{\sigma ^{\lambda
+\frac{j}{2}+\frac{1}{2}}}\leq \frac{C}{(\sin \varphi )^{2\lambda
+1}}.
\end{eqnarray*}

Thus, we have obtained that
\begin{equation}\label{acotK}
\sum_{\beta =1}^3|K_{\lambda }^{k,\beta} (\theta ,\varphi )|\leq
\frac{C}{(\sin \varphi )^{2\lambda +1}}\left(1+\sqrt{\frac{\sin
\varphi }{|\theta -\varphi |}}\right).
\end{equation}

By considering (\ref{R}), (\ref{I3}), (\ref{K}) and the estimations
(\ref{acotI1}), (\ref{acotI2}), (\ref{acotJk}) and (\ref{acotK}) we
conclude that, when $\frac{\theta}{2}\le\varphi
\le\frac{3\theta}{2}$, $\theta\neq\varphi$,
$$
\left|R_\lambda ^k(\theta ,\varphi )-\frac{R^k(\theta,\varphi
)}{\sigma ^\lambda}\right|\leq  \frac{C}{(\sin \varphi )^{2\lambda
+1}}\left(1+\sqrt{\frac{\sin \varphi }{|\theta -\varphi |}}\right),
$$
and the proof of Lemma \ref{lema} is finished.
\end{proof}

{\bf Step 3.} We now establish that the $k$-th Riesz transform in
the circle is a principal value integral operator, that is,
\begin{multline}\label{Rkpv}
\frac{d^k}{d\theta^k}\int_0^\pi f(\varphi)\int_0^1
\left(\log\frac{1}{r}\right)^{k-1}
\left(\frac{1-r^2}{1-2r\cos(\theta-\varphi)+r^2}-1\right)\frac{dr}{r} d\varphi\\
=2\pi \Gamma (k)\lim_{\varepsilon \rightarrow
0^+}\int_{0,|\theta-\varphi| > \varepsilon}^\pi f(\varphi)R^k(\theta
,\varphi )d\varphi + \beta_kf(\theta),\hspace{3cm}
\end{multline}
for every $\theta \in (0,\pi)$, and where $\beta_k=0$ when $k$ is
odd, and $\beta_k=2\pi (-1)^{\frac{k}{2}}\Gamma (k)$, when $k$ is even.

Let us consider the function
$$
H^k(\omega)=\int_0^1\left(\log\frac{1}{r}\right)^{k-1}\frac{\partial^{k-1}}{\partial\omega^{k-1}}\left(\frac{1-r^2}{1-2r\cos\omega+r^2}-1\right)
\frac{dr}{r},\quad \omega \in \mathbb{R} \setminus \{2k\pi:k\in
\mathbb{Z}\}.
$$
Firstly we are going to analyze the
behavior of $H^k(w)$ when $w \rightarrow 0^+$. We have that
\begin{eqnarray*}
H^1(w) &=& \int_0^1 \left(\frac{1-r^2}{1-2r\cos w + r^2}-1\right)\frac{dr}{r}  \\
&=&- \log(2(1-\cos w)),\quad  w \in \mathbb{R}\setminus \{2k\pi:k\in
\mathbb{Z}\}.
\end{eqnarray*}
Then, $\displaystyle \lim_{w \rightarrow 0} w H^1(w)=0$.

Assume that $k \in \mathbb N$, $k \geq 2$. According to
(\ref{derivDr}), it has
\begin{eqnarray*}
H^k(w) &=&  \sum_{\begin{subarray}{c}s=1,\ldots,k-1\\j \geq
2s-k+1\\i+j=s
                                     \end{subarray}} c_{k-1,s,i,j}\left(\int_0^{\frac{1}{2}}+ \int_{\frac{1}{2}}^1 \right)\frac{r^{i+j}(\cos w)^i (-\sin w)^j}{(1-2r\cos w + r^2)^{s+1}}(1-r^2)\left(\log\frac{1}{r} \right)^{k-1} \frac{dr}{r} \\
&=&  \sum_{\begin{subarray}{c}s=1,\ldots,k-1\\j \geq 2s-k+1\\i+j=s
                                     \end{subarray}} c_{k-1,s,i,j} \left(I^0_{k,s,i,j}(w) + I^1_{k,s,i,j}(w)\right), \quad w \in \mathbb{R}\setminus \{2k\pi:k\in
\mathbb{Z}\}.
\end{eqnarray*}
Let $s=1,\ldots, k-1$, $i+j=s$, $j \geq 2s-k+1$. By using the
dominated convergence theorem we obtain
$$
\lim_{w \rightarrow 0}I_{k,s,i,j}^0(w)= \left\{\begin{array}{ll}
                                              0, & j \geq 1;\\
                                              \displaystyle\int_{0}^{1/2}\left(\log\frac{1}{r}\right)^{k-1}\frac{(1-r^2)r^{s-1}}{(1-r)^{2(s+1)}}dr, & j=0;
                                              \end{array}\right.
$$
and, when $2s-k+1 < 0$,
$$
\lim_{w \rightarrow 0}I_{k,s,i,j}^1(w)= \left\{\begin{array}{ll}
                                              0, & j \geq 1;\\
                                              \displaystyle\int_{1/2}^{1}\left(\log\frac{1}{r}\right)^{k-1}\frac{(1-r^2)r^{s-1}}{(1-r)^{2(s+1)}}dr, & j=0.
                                              \end{array}\right.
$$
Assume now $2s-k+1 \geq 0$. We have that
$$
\left|I^1_{k,s,i,j}(w) \right| \leq
C|w|^{j+k-2s-1}\int_0^{\frac{1}{2|w|}} \frac{u^k}{(1+u^2)^{s+1}}du
\leq C|w|^{j+k-2s-\frac{3}{2}},\quad  w \in (-\pi,\pi)\setminus \{0\}.
$$
Then,
$$
\lim_{w \rightarrow 0} w I^1_{k,s,i,j}(w)=0,
$$
and if  $j>2s-k+1$,
$$
\lim_{w\rightarrow 0} I^1_{k,s,i,j}(w) = 0.
$$

Also, if $j=2s-k+1>0$ (as will be the case if k is even), by using
mean value theorem it follows
\begin{eqnarray*}
\lim_{w \rightarrow 0^+} I^1_{k,s,i,j}(w) &=& -2\lim_{w\rightarrow 0^+}\left(\frac{\sin w}{w}\right)^{2s-k+1}\int_0^{\frac{1}{2w}} \frac{u^k}{(1+u^2)^{s+1}}du \\
&=& -2\int_0^\infty \frac{u^{k}}{(1+ u^2)^{s+1}}du=-B\Big(\frac{k+1}{2},s-\frac{k-1}{2}\Big),
\end{eqnarray*}
where $B(x,y)$, $x,y>0$, represents the Beta Euler's function.

By combining the above estimates we conclude that $\displaystyle \lim_{w\rightarrow 0}wH^k(w)=0$, when $k$ is odd. Assume now that $k$ is even. In this case we obtain that
\begin{eqnarray*}
\lim_{w\rightarrow 0^+}H^k(w)&=&\sum_{s=1}^{\frac{k}{2}-1}c_{k-1,s,s,0}\int_0^1\Big(\log \frac{1}{r}\Big)^{k-1}\frac{(1-r^2)r^{s-1}}{(1-r)^{2(s+1)}}dr\\
&-&\sum_{s=\frac{k}{2}}^{k-1}c_{k-1,s,k-s-1, 2s-k+1}B\Big(\frac{k+1}{2},s-\frac{k-1}{2}\Big).
\end{eqnarray*}

By taking into account (\ref{FaadiBruno}) and the duplication formula for the Gamma Euler's function, we can write
\begin{eqnarray*}
\lim_{w\rightarrow 0^+}H^k(w)&=&-\sum_{s=\frac{k}{2}}^{k-1}\frac{(-1)^{s+1}(k-1)!s!}{2^{k-2s-1}(2s-k+1)!(k-s-1)!}B\Big(\frac{k+1}{2},s-\frac{k-1}{2}\Big)\\
&\hspace{-5cm}=&\hspace{-2.5cm}\frac{\pi (\Gamma (k))^2}{2^{k-2}\Gamma (\frac{k}{2})}\sum_{s=\frac{k}{2}}^{k-1}\frac{(-1)^s}{(2s-k+1)(k-s-1)!(s-\frac{k}{2})!}=\frac{(-1)^{\frac{k}{2}}\pi (\Gamma (k))^2}{2^{k-2}(\Gamma (\frac{k}{2}))^2}\sum_{r=0}^{\frac{k}{2}-1}(-1)^r {{\frac{k}{2}-1}\choose {r}}\frac{1}{2r+1}\\
&\hspace{-5cm}=&\hspace{-2.5cm}\frac{(-1)^{\frac{k}{2}}\pi (\Gamma (k))^2}{2^{k-2}(\Gamma (\frac{k}{2}))^2}\int_0^1(1-t^2)^{\frac{k}{2}-1}dt=(-1)^{\frac{k}{2}}\pi \Gamma (k).
\end{eqnarray*}

By proceeding as in Step 1 we can see that
\begin{eqnarray}\label{k-1}
\frac{d^{k-1}}{d\theta^{k-1}}\int_0^\pi f(\varphi)\int_0^1
\left(\log\frac{1}{r}\right)^{k-1}\left(\frac{1-r^2}{1-2r\cos(\theta-\varphi)+r^2}-1\right)\frac{dr}{r} d\varphi\nonumber\\
&\hspace{-20cm}=&\hspace{-10cm} \int_0^\pi f(\varphi)\int_0^1
\left(\log\frac{1}{r}\right)^{k-1}\frac{\partial^{k-1}}{\partial\theta^{k-1}}\left(\frac{1-r^2}{1-2r\cos(\theta-\varphi)+r^2}-1\right)\frac{dr}{r}
d\varphi, \quad \theta \in (0,\pi).
\end{eqnarray}
It is clear that
\begin{multline*}
\frac{d^{k-1}}{d\theta^{k-1}} \int_0^\pi f(\varphi)\int_0^1 \left(\log\frac{1}{r}\right)^{k-1}\left(\frac{1-r^2}{1-2r\cos(\theta-\varphi)+r^2}-1 \right)\frac{dr}{r} d\varphi \\
= \int_{-\pi}^{\pi} \tilde{f}(\varphi)H^k(\theta-\varphi)\;d\varphi, \quad  \theta \in (0,\pi),\hspace{3cm}
\end{multline*}
where $ \tilde{f}(\varphi) = \left\{\begin{array}{ll}
                                        f(\varphi), &  \varphi \in [0,\pi) \\
                                        0, & \varphi \in (-\pi,0)
                                        \end{array} \right. $,
and  $\tilde{f}(\varphi) = \tilde{f}(\varphi+2\pi)$, $\varphi \in \mathbb R$.

Also, since $H^k\in L^1(-\pi,\pi)$, we have that,
\begin{multline*}
\frac{d}{d\theta}\int_{-\pi}^\pi \tilde{f}(\varphi) H^k(\theta-\varphi)d\varphi
= \frac{\partial}{\partial\theta}\int_{-\pi}^\pi \tilde{f}(\theta-u) H^k(u)du =\int_{-\pi}^\pi
\frac{\partial}{\partial\theta}\tilde{f}(\theta-u) H^k(u)du\\
\shoveleft{\hspace{15mm}= -\lim_{\varepsilon \rightarrow 0^+} \int_{-\pi, |u|>\varepsilon}^\pi \frac{d}{d u}[\tilde{f}(\theta-u)] H^k(u)du }\\
\shoveleft{\hspace{15mm}=-\lim_{\varepsilon \rightarrow
0^+}\left(\tilde{f}(\theta-u) H^k(u)\right]_{\varepsilon}^\pi
- \int_\varepsilon^\pi \tilde{f}(\theta-u)\frac{d}{d u}H^k(u) du }\\
\shoveleft{\hspace{27mm}+ \left. \left. \tilde{f}(\theta-u)
H^k(u)\right]_{-\pi}^{-\varepsilon}
- \int^{-\varepsilon}_{-\pi} \tilde{f}(\theta-u)\frac{d}{d u}H^k(u) du \right)}\\
\shoveleft{\hspace{15mm}=-\lim_{\varepsilon \rightarrow 0^+}\left(
\tilde{f}(\theta-\pi)H^k(\pi)
-\tilde{f}(\theta-\varepsilon)H^k(\varepsilon) + \tilde{f}(\theta+\varepsilon)H^k(-\varepsilon) \right.}\\
\ - \tilde{f}(\theta+\pi)H^k(-\pi)- \int_{-\pi,|u|>\varepsilon}^\pi
\tilde{f}(\theta-u)\frac{d}{d u}H^k(u) du \Big),\quad \theta \in
(0,\pi ).\hspace{13mm}
\end{multline*}
Since the function $H^k$ is even when $k$ is odd and $H^k$ is odd when $k$ is even, we conclude that
\begin{multline*}
\frac{d}{d\theta}\int_0^\pi f(\varphi)H^k(\theta-\varphi) d\varphi = \lim_{\varepsilon \rightarrow 0^+}\int_{-\pi,|\theta-\varphi|>\varepsilon}^\pi \tilde f(\varphi)
 \left(\frac{d}{du}H^k\right)(\theta-\varphi)d\varphi\\
\shoveleft{\hspace{3mm}
 -\lim_{\varepsilon \rightarrow
0^+}(f(\theta+\varepsilon)-f(\theta-\varepsilon))H^k(\varepsilon)= \lim_{\varepsilon \rightarrow 0^+}
\int_{0,|\theta-\varphi|>\varepsilon}^\pi f(\varphi)
\frac{\partial}{\partial \theta}
H^k(\theta-\varphi)d\varphi,\quad \theta \in (0,\pi ), } \\
%\hspace{6cm}
\end{multline*}
when $k$ is odd, and
\begin{multline*}
\frac{d}{d\theta}\int_0^\pi f(\varphi)H^k(\theta-\varphi)d\varphi \\
=\lim_{\varepsilon \rightarrow 0^+} (f(\theta +
\varepsilon)+f(\theta - \varepsilon))H^k(\varepsilon) +
\lim_{\varepsilon \rightarrow 0^+} \int_{0,|\theta - \varphi|>
\varepsilon}^\pi f(\varphi)\frac{\partial}{\partial
\theta}H^k(\theta-\varphi)d\varphi \\
=2f(\theta )(-1)^{\frac{k}{2}}\pi \Gamma (k) +
\lim_{\varepsilon \rightarrow 0^+} \int_{0,|\theta - \varphi|>
\varepsilon}^\pi f(\varphi)\frac{\partial}{\partial
\theta}H^k(\theta-\varphi)d\varphi ,\quad \theta \in (0,\pi
),\hspace{15mm}
\end{multline*}
when $k$ is even.

{\bf Step 4.} We now finish the proof of Theorem \ref{thprin}. We
firstly write, according to (\ref{lderivada}),
\begin{multline*}
\frac{d^{k-1}}{d\theta^{k-1}}L^{-\frac{k}{2}}_\lambda
f(\theta)=\int_0^\pi
f(\varphi)R^{k,k-1}_\lambda(\theta,\varphi)dm_\lambda(\varphi)=\int_0^\pi
f(\varphi)\frac{R^{k,k-1}(\theta,\varphi)}{(\sin\theta\sin\varphi)^\lambda}dm_\lambda(\varphi)\\
\hspace{2cm}+\int_0^\pi f(\varphi)\left(R_\lambda
^{k,k-1}(\theta,\varphi)-\frac{R^{k,k-1}(\theta,\varphi)}{(\sin\theta\sin\varphi)^\lambda}
\right)dm_\lambda(\varphi),
\quad \theta \in (0,\pi),\\
\end{multline*}
where, for every $\theta,\varphi \in (0,\pi)$,
$$
R^{k,k-1}(\theta,\varphi)=\frac{1}{2\pi \Gamma (k)}\int_0^1
\left(\log\frac{1}{r}\right)^{k-1}\frac{\partial^{k-1}}{\partial\theta^{k-1}}\left(\frac{1-r^2}{1-2r\cos(\theta-\varphi)+r^2}-1\right)\frac{dr}{r}.
$$
Moreover, by (\ref{Rkpv}) and (\ref{k-1}) we have, for every $\theta
\in (0,\pi )$,
\begin{multline}\label{acotacion1VP}
\frac{d}{d\theta}\int_0^\pi
f(\varphi)\frac{R^{k,k-1}(\theta,\varphi)}{(\sin\theta\sin\varphi)^\lambda}\;dm_\lambda(\varphi)=\frac{1}{2\pi\Gamma(k)}\frac{d}{d\theta}\left(\frac{1}{(\sin\theta)^\lambda}\right.\\
\shoveleft{\hspace{5mm}\left. \times 
\frac{d^{k-1}}{d\theta^{k-1}}\int_0^\pi
\frac{f(\varphi)}{(\sin\varphi)^\lambda}\int_0^1
\left(\log\frac{1}{r}\right)^{k-1}
\left(\frac{1-r^2}{1-2r\cos(\theta-\varphi)+r^2}-1\right)\frac{dr}{r}dm_\lambda(\varphi)\right)}\\
\shoveleft{\hspace{-1cm}=-\frac{\lambda\cos\theta}{(\sin\theta)^{\lambda+1}}\int_0^\pi\frac{f(\varphi)}{(\sin\varphi)^\lambda}
R^{k,k-1}(\theta,\varphi)dm_\lambda(\varphi)}+\frac{1}{2\pi\Gamma(k)}
\frac{1}{(\sin\theta)^\lambda}\\
\times \frac{d^k}{d\theta^k}\int_0^\pi
\frac{f(\varphi)}{(\sin\varphi)^\lambda}\int_0^1
\left(\log\frac{1}{r}\right)^{k-1}
\left(\frac{1-r^2}{1-2r\cos(\theta-\varphi)+r^2}-1\right)\frac{dr}{r}dm_\lambda(\varphi)\hspace{1cm}
\end{multline}
\begin{multline*}
=-\frac{\lambda\cos\theta}{(\sin\theta)^{\lambda+1}}\int_0^\pi\frac{f(\varphi)}{(\sin\varphi)^\lambda}
R^{k,k-1}(\theta,\varphi)dm_\lambda(\varphi)\\
+\frac{1}{(\sin\theta)^\lambda}\lim_{\varepsilon \rightarrow
0^+}\int^\pi_{0,|\theta-\varphi|>\varepsilon}
\frac{f(\varphi)}{(\sin\varphi)^\lambda}R^k(\theta ,\varphi
)dm_\lambda(\varphi)+\gamma _kf(\theta
),\hspace{4.5cm}
\end{multline*}
where the integral after the last equal sign is absolutely convergent and $\gamma _k=0$, if $k$ is odd and $\gamma _k=(-1)^{\frac{k}{2}}$, when $k$ is even.

A careful study of Lemma \ref{lema} and again a distributional
argument allow us to justify the differentiation under the integral
sign (see Lemma \ref{derint1} in Appendix) to get
\begin{multline}\label{acotacion2VP}
\frac{d}{d \theta}\!\!\int_0^\pi \!\!\!\!f(\varphi)\!
\left(\!R_\lambda^{k,k-1}(\theta,\varphi)\!-\!
\frac{R^{k,k-1}(\theta,\varphi)}{(\sin\theta\sin\varphi)^\lambda}\!\right)\!dm_\lambda(\varphi)\!=\!\frac{\lambda \cos \theta
}{(\sin\theta)^{\lambda+1}}\int_0^{\pi}\!\!\frac{f(\varphi)}{(\sin\varphi)^\lambda}
R^{k,k-1}(\theta,\varphi)dm_\lambda(\varphi)\\
+\int_0^{\pi}f(\varphi)\left(R_\lambda^k(\theta,\varphi)-
\frac{R^k(\theta,\varphi)}{(\sin\theta\sin\varphi)^\lambda}\right) dm_\lambda(\varphi),\quad \mbox{a.e. }\theta
\in(0,\pi),\hspace{3cm}
\end{multline}
where all the integrals are absolutely convergent.

By combining (\ref{acotacion1VP}) and (\ref{acotacion2VP}) we conclude that
\begin{eqnarray*}
\frac{d^k}{d \theta^k}L_\lambda ^{-\frac{k}{2}}f(\theta) &=&
\frac{d}{d\theta}\int_0^\pi
f(\varphi)\Big(R_\lambda^{k,k-1}(\theta,\varphi)-\frac{R^{k,k-1}(\theta,\varphi)}{(\sin\theta \sin\varphi)^\lambda}\Big)dm_\lambda(\varphi)\\
&\hspace{-4cm}+&\hspace{-2cm}\frac{d}{d\theta}\int_0^\pi f(\varphi)\frac{R^{k,k-1}(\theta,\varphi)}{(\sin\theta \sin\varphi)^\lambda}dm_\lambda(\varphi)\\
&\hspace{-5cm}=&\hspace{-2.5cm}\lim_{\varepsilon \rightarrow
0^+}\int_{0,|\theta-\varphi|>\varepsilon}^\pi f(\varphi )
R_\lambda^k(\theta,\varphi)dm_\lambda(\varphi)-\lim_{\varepsilon
\rightarrow 0^+}\int_{0,|\theta-\varphi|>\varepsilon}^\pi f(\varphi
)\frac{R^{k}(\theta,\varphi)}{(\sin\theta \sin\varphi)^\lambda}dm_\lambda(\varphi)\\
&\hspace{-4cm}+&\hspace{-2cm}\lambda\frac{\cos\theta}{(\sin\theta)^{\lambda+1}}\int_0^\pi
\frac{f(\varphi)}{(\sin\varphi)^\lambda}R^{k,k-1}(\theta,\varphi)dm_\lambda(\varphi)+\frac{d}{d\theta}\int_0^\pi f(\varphi)\frac{R^{k,k-1}(\theta,\varphi)}{(\sin\theta \sin\varphi)^\lambda}dm_\lambda(\varphi)\\
&\hspace{-5cm}=&\hspace{-2.5cm}\lim_{\varepsilon \rightarrow
0^+}\int_{0,|\theta-\varphi|>\varepsilon}^\pi f(\varphi )
R_\lambda^k(\theta,\varphi)dm_\lambda(\varphi)+\gamma_kf(\theta),\quad \mbox{ a.e. }\theta
\in (0,\pi).
\end{eqnarray*}
Thus the proof of Theorem \ref{thprin} is complete.

\section{Proof of Theorem \ref{oscilacion}}

In order to show Theorem \ref{oscilacion} we need to improve Lemma
\ref{lema} as follows,
\begin{lema}\label{lemosc}
Let $k\in\mathbb N$. If $R^k$ and $A_2$ are defined as in Lemma
\ref{lema}, then
$$
R^k(\theta,\varphi) = M_k\frac{1}{\sin (\theta-\varphi)} +
O\left(\sqrt{\frac{1}{|\theta - \varphi|}}\right),\;\;
(\theta,\varphi) \in A_2,\quad  \theta \neq \varphi,
$$
for a certain $M_k \in \mathbb R$. Moreover, $M_k=0$ provided that
$k$ is even.
\end{lema}
\begin{proof} According to (\ref{derivDr}) with $\lambda=t=0$ we have that
$$
\frac{\partial}{\partial \theta^k}\left(\frac{1}{\Delta_r}\right) =
\sum_{s,i,j} c_{k,s,i,j}\frac{r^{i+j}a^ib^j}{\Delta_r^{1+s}},
$$
where $a=\cos(\theta-\varphi)$, $b=-\sin(\theta-\varphi)$, and
$c_{k,s,i,j} \neq 0$ only if $s=1,\ldots,k$, $j \geq 2s-k$ and
$i+j=s$.

Note firstly that
$$
\int_0^{\frac{1}{2}}
\left(\log\frac{1}{r}\right)^{k-1}(1-r^2)\left|\frac{\partial}{\partial
\theta^k}\left(\frac{1}{1+r^2-2r\cos(\theta-\varphi)}\right)\right|
\frac{dr}{r} \leq C, \quad \theta,\varphi \in (0,\pi).
$$
Also, we have
\begin{eqnarray*}
\int_{\frac{1}{2}}^1 \left(\log\frac{1}{r}\right)^{k-1}(1-r^2)\frac{r^{i+j}|a|^i|b|^j}{\Delta_r^{1+s}}\frac{dr}{r} &\leq& C\int_{\frac{1}{2}}^1 \frac{(1-r)^k|b|^j}{((1-r)^2 + \Delta)^{1+s}}dr \\
&\leq& C \frac{|b|^j}{\Delta^{s-\frac{k}{2}+\frac{1}{2}}}\int_0^{\frac{1}{2\sqrt{\Delta}}}\frac{u^k}{(1+u)^{2+2s}}du\\
&\leq& C\frac{|b|^j}{\Delta^{\frac{j}{2}+\frac{1}{4}}}\int_0^{\frac{1}{2\sqrt{\Delta}}}\frac{u^{2s-j+\frac{1}{2}}}{(1+u)^{2+2s}}du \\
&\leq&C\frac{1}{\sqrt{|\theta-\varphi|}},\quad  (\theta,\varphi) \in
A_2,\;\; \theta \neq \varphi,
\end{eqnarray*}
provided that $s=1,\ldots,k$, $i+j=s$,  $j > 2s-k$.

Assume now $s=1,\ldots,k$, $i+j=s$,  $j= 2s-k$. By using the mean
value theorem we get
\begin{multline*}
\int_{\frac{1}{2}}^1\left(\log\frac{1}{r}\right)^{k-1}(1-r^2)\frac{r^{i+j}a^ib^j}{\Delta_r^{1+s}}\frac{dr}{r} \\
= 2\int_{\frac{1}{2}}^1 \frac{(1-r)^k}{((1-r)^2+\Delta)^{1+s}}dr
a^ib^j +
O\left(\frac{1}{\sqrt{|\theta-\varphi|}}\right),\;\;\theta,\varphi\in
A_2,\;\; \theta \neq \varphi.
\end{multline*}

Moreover, since $2s-k\ge 0$,
$$
\int_{\frac{1}{2}}^1 \frac{(1-r)^k}{((1-r)^2+\Delta)^{1+s}}dr = \frac{1}{\Delta^{s-\frac{k}{2}+\frac{1}{2}}}\left(\int_0^\infty \frac{u^kdu}{(1+u^2)^{1+s}} - \int_{\frac{1}{2\sqrt{\Delta}}}^\infty \frac{u^kdu}{(1+u^2)^{1+s}}\right)
$$
and
\begin{eqnarray*}
\frac{|a|^i|b|^j}{\Delta^{s-\frac{k}{2}+\frac{1}{2}}}\int_{\frac{1}{2\sqrt{\Delta}}}^\infty \frac{u^kdu}{(1+u^2)^{1+s}} &\leq& C \frac{1}{\Delta^{1/4}}\int_{\frac{1}{2\sqrt{\Delta}}}^\infty \frac{u^{k+\frac{1}{2}}du}{(1+u)^{2+s}} \\
&\leq& \frac{C}{\Delta^{\frac{1}{4}}} \leq
\frac{C}{|\theta-\varphi|^{\frac{1}{2}}},\quad \theta,\varphi \in
A_2,\;\theta \neq \varphi.
\end{eqnarray*}
Also, if $k$ is odd, we have that
$$
\frac{a^ib^j}{\Delta^{s-\frac{k}{2}+\frac{1}{2}}} =
\frac{(\cos(\theta-\varphi))^i(-\sin(\theta-\varphi))^j}{(2(1-\cos(\theta-\varphi)))^{\frac{j}{2}+\frac{1}{2}}}
=- \frac{1}{\sin(\theta-\varphi)} +
O\left(\frac{1}{|\theta-\varphi|^{\frac{1}{2}}}\right), \quad
\theta,\varphi \in A_2, \; \theta \neq \varphi.
$$
By combining the above estimates we conclude that
$$
R^k(\theta,\varphi) = M_k\frac{1}{\sin(\theta-\varphi)} +
O\left(\frac{1}{|\theta-\varphi|^{\frac{1}{2}}}\right), \quad
\theta,\varphi \in A_2, \; \theta \neq \varphi,
$$
for a certain $M_k \in \mathbb R$, for every $k$ odd.

Assume now that $k$ is even. We get
$$
\frac{a^ib^j}{\Delta^{s-\frac{k}{2}+\frac{1}{2}}} =
\frac{1}{|\sin(\theta-\varphi)|} +
O\left(\frac{1}{|\theta-\varphi|^{\frac{1}{2}}}\right), \quad
\theta,\varphi \in A_2, \; \theta \neq \varphi.
$$
Hence, from Lemma \ref{lema} we deduce that, for every $\theta,\varphi
\in A_2, \; \theta \neq \varphi$,
$$
R_\lambda^k(\theta,\varphi)=M_k\frac{1}{|\sin(\theta-\varphi)|(\sin\theta\sin\varphi)^\lambda}+O\Big(\frac{1}{(\sin\theta\sin\varphi)^{\lambda+1/2}}
\Big(1+\sqrt{\frac{\sin\theta}{|\theta-\varphi|}}\Big)\Big).
$$
By virtue of Theorem \ref{thprin}, $M_k=0$ because
$$
\lim_{\varepsilon\to
0^+}\int_{\theta/2,\,|\theta-\varphi|>\varepsilon}^{3\theta/2}\frac{1}{|\sin(\theta-\varphi)|}(\sin\varphi)^\lambda
d\varphi,
$$
does not exist for every $\theta\in (0,\pi/2)$.
\end{proof}

From Lemmas \ref{lema} and \ref{lemosc} we deduce that,
\begin{equation}\label{Rtcomp}
R_\lambda^k(\theta,\varphi)\!=\!\left\{\!\!\begin{array}{l}
                                    O\left((\sin\varphi)^{-(2\lambda+1)}\right),\;\; (\theta,\varphi) \in
                                    A_1;\\
                                    \displaystyle \frac{M_k}{(\sin\theta\sin\varphi)^\lambda \sin(\theta-\varphi)}+ O\!\left(\!\frac{1}{(\sin\varphi)^{2\lambda+1}}\!\left(\!\!1+\!\sqrt{\frac{\sin\varphi}{|\theta-\varphi|}}\!\right)\!\!\right)\!, \; (\theta,\varphi)\in A_2,\;\theta \neq \varphi;\\
                                    O\left((\sin\theta)^{-(2\lambda+1)}\right),\;\; (\theta,\varphi) \in A_3.\\
                                    \end{array}\right.
\end{equation}
By using (\ref{Rtcomp}) we can prove Theorem \ref{oscilacion} by
proceeding as in the proof of \cite[Proposition 8.1]{BMTU}.

\section{Appendix}
In this appendix we present the results we need about differentiation under the integral sign. We think that these results are wellknown but we have not found a exact reference (only the unpublished notes \cite{Cheng}). Then we prefer to include here a  proof of the result in the form we use.
We look for conditions on a function $f$ defined on $\mathbb{R}\times \mathbb{R}$ in order that the formula
$$
\frac{\partial}{\partial x}\int_\mathbb{R}f(x,y)dy=\int_\mathbb{R}\frac{\partial}{\partial x} f(x,y)dy,\quad \mbox{ a.e. }x\in \mathbb{R},
$$
holds.

In the following we establish conditions on a function $f$ in order that distributional and classical derivatives of $f$ coincide.

\begin{lema}\label{derint2}
Let $-\infty \leq a<b\leq +\infty$. Assume that $f$ is a continuous function on $I\times I$, where $I=(a,b)$, such that

(i) For every $y\in I$, the function $\frac{\partial}{\partial x}f(x,y)dy$ is continuous on $I\setminus \{y\}$, where the derivative is understood in the classical sense.

(ii) For every $y\in I$ and every compact subset $K$ of $I$, $\displaystyle\int_K|f(x,y)|dx<+\infty$, and $$\displaystyle\int_K\left|\frac{\partial f}{\partial x}(x,y)\right|dx<+\infty. $$

Then, $D_xf(x,y)=\frac{\partial}{\partial x}f(x,y)$, for every $y\in I$. Here, as above, $D_xf(x,y)$ denotes the distributional derivative respect to $x$ of $f$.
\end{lema}

\begin{proof}
Let $g \in C_c^\infty (I)$. We can write
\begin{eqnarray*}
<D_xf(x,y),g (x)>&=&-\lim_{\varepsilon \rightarrow 0^+}\left(\int_a^{y-\varepsilon}+\int_{y+\varepsilon }^b\right)g'(x)f(x,y)dx\\
&\hspace{-6cm}=&\hspace{-3cm}\lim_{\varepsilon \rightarrow 0^+}\left[-g(y-\varepsilon )f(y-\varepsilon ,y)+
g(y+\varepsilon )f(y+\varepsilon ,y)+\left(\int_a^{y-\varepsilon}+\int_{y+\varepsilon }^b\right) g(x)\frac{\partial f}{\partial x}(x,y)dx\right]\\
&\hspace{-6cm}=&\hspace{-3cm}\int_a^bg(x)\frac{\partial f}{\partial x}(x,y)dx,\quad y\in I.
\end{eqnarray*}
Then, $D_xf(x,y)=\frac{\partial f}{\partial x}(x,y)$, $y\in I$.
\end{proof}

The differentiations under the integral sign that we have made in the proof of our results can be justified by using the following one.

\begin{lema}\label{derint1}
Suppose that $f$ is a measurable function defined on $\mathbb{R}\times \mathbb{R}$ that satisfies the fo\-llo\-wing conditions:

(i) for every compact subset $K$ of $\mathbb{R}$, $\int_K\int_\mathbb{R}|f(x,y)|dydx<\infty $, and

(ii) there exists a measurable function $g$ on $\mathbb{R}\times \mathbb{R}$ such that
$\int_K\int_\mathbb{R}|g(x,y)|dydx<\infty$, for every compact subset $K$ of $\mathbb{R}$, and that the distributional derivative
$D_xf(\cdot , y)$ is represented by $g(\cdot, y)$, for every $y\in \mathbb{R}$.

Then,
$$
\frac{\partial}{\partial x}\int_\mathbb{R}f(x,y)dy=\int_\mathbb{R}\frac{\partial }{\partial x}f(x,y)dy,\quad \mbox{ a.e. }x\in \mathbb{R},
$$
where the derivatives are understood in the classical sense.
\end{lema}

\begin{proof} We define the function $h(x)=\int_\mathbb{R}f(x,y)dy$, $x\in \mathbb{R}$. By (i) $h$ defines a regular distribution that
we  continue denoting by $h$. According to \cite[Chap. 2, \S 5, Theorem V]{Schw}, we have that
$$
\frac{\partial}{\partial x}f(x,y)=g(x,y),\quad \mbox{ a.e. }(x,y)\in \mathbb{R}\times\mathbb{R},
$$
where the derivative is understood in the classical sense. Moreover, if $F \in C_c^\infty (\mathbb{R})$, then
$$
<D_xh,F> = \int_\mathbb{R}F(x)\int_\mathbb{R}\frac{\partial f}{\partial x}(x,y)dydx.
$$

Hence, $D_xh(x)=\int_\mathbb{R}\frac{\partial f}{\partial x}(x,y)dy$ in the distributional sense. By using again \cite[Chap. 2, \S 5, Theorem V]{Schw} we conclude that
$$
\frac{\partial}{\partial x}h(x)=\int_\mathbb{R}\frac{\partial }{\partial x}f(x,y)dy,\quad\mbox{ a.e. }x\in \mathbb{R}.
$$
Thus the proof is completed.
\end{proof}

%\bibliography{bibliobasevpg}
%\bibliographystyle{amsplain}

\end{document}